\documentclass[10pt,twoside,a4paper]{article}

\usepackage[utf8]{inputenc}

\usepackage[T1]{fontenc}

\usepackage{amsmath}
\numberwithin{equation}{section}

\usepackage{amssymb}

\usepackage{microtype}

\usepackage[backend=biber,maxnames=15]{biblatex}
\DefineBibliographyStrings{english}{%
	backrefpage = {cited on page},
	backrefpages = {cited on pages},
}

\DeclareFieldFormat[report]{title}{``#1''}
\DeclareFieldFormat[book]{title}{``#1''}
\AtEveryBibitem{\clearfield{url}}
\AtEveryBibitem{\clearfield{note}}

\usepackage{graphicx}
\graphicspath{{images/}}

\usepackage{tikz}
\usepackage{pgfplots}
\pgfplotsset{compat=newest}

\title{Quantitative Parameter Reconstruction from Optical Coherence Tomographic Data}

\author{Leopold Veselka$^1$\\{\footnotesize\href{mailto:leopold.veselka@univie.ac.at}{leopold.veselka@univie.ac.at}}
\and Peter Elbau$^1$\\{\footnotesize\href{mailto:peter.elbau@univie.ac.at}{peter.elbau@univie.ac.at}}
\and Leonidas Mindrinos$^3$\\{\footnotesize\href{mailto:leonidas.mindrinos@aua.gr}{leonidas.mindrinos@aua.gr}}
\and Lisa Krainz$^2$\\{\footnotesize\href{mailto:lisa.krainz@meduniwien.ac.at}{lisa.krainz@meduniwien.ac.at}}
\and Wolfgang Drexler$^2$\\{\footnotesize\href{mailto:wolfgang.drexler@meduniwien.ac.at}{wolfgang.drexler@meduniwien.ac.at}}}
\date{}

\usepackage[pdftex,colorlinks=true,linkcolor=blue,citecolor=green,urlcolor=blue,bookmarks=true,bookmarksnumbered=true]{hyperref}
\hypersetup
{
    pdfauthor={Leopold Veselka, Lisa Krainz, Leonidas Mindrinos, Wolfgang Drexler, Peter Elbau},
    pdfsubject={},
    pdftitle={Quantitative Parameter Reconstruction from Optical Coherence Tomography},
    pdfkeywords={}
}

\usepackage[hyperref,amsmath,thmmarks]{ntheorem}
\usepackage{aliascnt}

\newtheorem{lemma}{Lemma}[section]

\newaliascnt{proposition}{lemma}
\newtheorem{proposition}[proposition]{Proposition}
\aliascntresetthe{proposition}

\newaliascnt{corollary}{lemma}

\aliascntresetthe{corollary}

\newaliascnt{theorem}{lemma}
\newtheorem{theorem}[theorem]{Theorem}
\aliascntresetthe{theorem}

\newaliascnt{remark}{lemma}
\newtheorem{remark}[remark]{Remark}
\aliascntresetthe{remark}

\theorembodyfont{\normalfont}
\newaliascnt{definition}{lemma}

\aliascntresetthe{definition}

\newaliascnt{assumption}{lemma1}
\newtheorem{assumption}[assumption]{Assumption}
\aliascntresetthe{assumption}

\newaliascnt{method}{lemma}

\aliascntresetthe{method}

\theoremstyle{nonumberplain}
\theoremseparator{:}
\theoremheaderfont{\normalfont\itshape}

\theoremsymbol{\ensuremath{\square}}
\newtheorem{proof}{Proof}

\usepackage[a4paper,centering,bindingoffset=0cm,marginpar=2cm,margin=2.5cm]{geometry}

\usepackage[pagestyles]{titlesec}
\titleformat{\section}[block]{\large\sc\filcenter}{\thesection.}{0.5ex}{}[]
\titleformat{\subsection}[runin]{\bf}{\thesubsection.}{0.5ex}{}[.]

\newpagestyle{headers}%
{%
	\headrule
	\sethead%
		[\footnotesize\thepage]%
		[\footnotesize\sc Authors]%
		[]%
		{}%
		{\footnotesize\sc Title}%
		{\footnotesize\thepage}
	\setfoot{}{}{}
}

\usepackage[font=footnotesize,format=plain,labelfont=sc,textfont=sl,width=0.75\textwidth,labelsep=period]{caption}

\usepackage{subcaption}

\usepackage{siunitx}

\usepackage{dsfont}
\usepackage{bm}
\newcommand{\N}{\mathds{N}}

\newcommand{\R}{\mathds{R}}
\newcommand{\C}{\mathds{C}}

\let\RE\Re
\let\Re=\undefined
\DeclareMathOperator{\Re}{\RE e}
\let\IM\Im
\let\Im=\undefined
\DeclareMathOperator{\Im}{\IM m}
\DeclareMathOperator{\supp}{supp}

\DeclareMathOperator{\argmin}{argmin}

\DeclareMathOperator{\sign}{sgn}
\DeclareMathOperator{\sincs}{si}

\newcommand{\un}[1]{\underline{#1}}

\newcommand{\lk}[1]{\underline{#1}}
\newcommand{\bk}[1]{\bar{#1}}

\newcommand{\e}{\mathrm e}

\usepackage{xcolor}
\usepackage{chngcntr}
\counterwithout{equation}{section}

\begin{document}
	
\maketitle
\thispagestyle{empty}
\begin{center}
\hspace*{5em}
\parbox[t]{12em}{\footnotesize
\hspace*{-1ex}$^1$University of Vienna\\
Oskar-Morgenstern-Platz 1\\
A-1090 Vienna, Austria}\hfil 
\parbox[t]{17em}{\footnotesize
\hspace*{-1ex}$^2$Medical University of Vienna\\
Waehringer Guertel 18-20\\
A-1090 Vienna, Austria}

\parbox[t]{17em}{\footnotesize
\hspace*{-1ex}$^3$Agricultural University of Athens \\
Department of Natural Resources, Development and Agricultural Engineering\\
Athens, Greece}
\end{center}

\abstract{ 
Quantitative tissue information, like the light scattering properties, is considered as a key player in the detection of cancerous cells in medical diagnosis. A promising method to obtain these data is optical coherence tomography (OCT). In this article, we will therefore discuss the refractive index reconstruction from OCT data, employing a Gaussian beam based forward model. 
We consider in particular samples with a layered structure, meaning that the refractive index as a function of depth is well approximated by a piece-wise constant function. For the reconstruction, we present a layer-by-layer method where in every step the refractive index is obtained via a discretized least squares minimization. For an approximated form of the minimization problem, we present an existence and uniqueness result. \\
The applicability of the proposed method is then verified by reconstructing refractive indices of layered media from both simulated and experimental OCT data.            
}
\section{Introduction}\label{sec:introduction}

Optical coherence tomography is a non-invasive, high-precision imaging modality with micrometer resolution based on the interferometric measurement \cite{FerMenWer88} of backscattered light. Since its invention in the 1990's \cite{HuaSwaLinSchuStiCha91}, different optical coherence tomographic systems, see \cite{DreFuj08}, for example, have been developed. All share the basic working principle, which can be described as follows: 

Laser light in the near infrared region is sent into the system. A beamsplitter then separates the light in two parts. One is directed into the sample arm and the second into the reference arm. Depending on the type of the system, the object in the sample arm is then either raster scanned, meaning that depth profiles on a lateral grid across the object are obtained, or illuminated at once. In both cases, the backscattered light by the object is then coupled into the system again and transported to the detector. There, the combined intensity of the light from the sample and the reference arm, where the light is backreflected by a perfect mirror, is detected.   
   
Due to its (ultra-)high resolution property and its very fast acquisition rate of up to more than 50000 raster scans per second, OCT, together with its adaptions and its extensions like polarization-sensitive OCT (PS-OCT) \cite{Bau17,BoeMilGemNel97,HeeHuaSwaFuj92}, optical coherence angiography \cite{KhaMehKhaIqbRas17,LiuDre19} or optical coherence elastography (OCE) \cite{KenKenSam14}, has become an outstanding technology in imaging of biological tissues, especially in the field of ophthalmology.

Nowadays, the structural information which is provided by such an OCT system is already successfully used for medical diagnosis. However, the obtained information is mainly of qualitative nature. For medical diagnosis it is worth to have supplementary quantitative information, like the optical (scattering) properties of the object of interest, which are considered as future key makers in the field of medical diagnosis. 

The quantification of such physical properties, in our case the refractive index, from OCT data, commonly represents an ill-posed problem, which is interesting from a mathematical perspective. This is the scope of this article. In its most general form, the parameter quantification in (Fourier domain based) OCT is a severely ill-posed problem dealing with the reconstruction of the optical properties, which we describe by the refractive index which is a function of space (three dimensions) and wavenumber (one dimension), from maximally three-dimensional (two spatial dimensions and one for the wavenumber) measurement data \cite{ElbMinSch15}.
The insufficient amount of available data has driven the need for modeling possibilities of the direct problem, which are commonly expressed by assumptions on the object. A general overview of existing modeling possibilities for the direct problem in OCT is presented, for example, in \cite{ElbMinSch15} or \cite{ElbMinVes23}. 

The inverse problem in OCT can be seen as an inverse electromagnetic scattering problem, especially if one considers the ideal case of having direct access to the backscattered sample. It has been treated, also non-specifically related to OCT, in a theoretical context, for example, in \cite{ColKre98} and under a number of (simplifying) assumptions, like having a weakly scattering medium, in \cite{MacArrMun23}.    

The aim of this article is to provide a reconstruction method from experimental data obtained by a swept-source OCT system, a specific form of Fourier domain OCT, which is based on a raster scanning of the object with focused and practically monochromatic laser light centered at multiple wavelengths. Due to this raster scanning process, the (inverse) scattering problem can be restricted to the narrow illuminated region, where the refractive index is typically considered only as a depth dependent function.       

One of the very first attempts on reconstructing a depth dependent refractive index from an OCT depth profile by Fourier transform has been presented in \cite{Fer96}. The reconstruction under the, in this case, crucial assumption of a weakly scattering object (so that the Born approximation of the electromagnetic wave equation is valid) has also been treated in \cite{RalMarCarBop06}, for example, where additionally a Gaussian beam model for the incident laser light has been employed.

Another assumption on the object, which substantially simplifies the mathematical model, is that the object shows a multi-layer structure. This case typically can be identified in OCT images of the human retina \cite{HeeIzaSwaHuaSchu95} and in human skin imaging. For such a multi-layer structure, at least locally in the illuminated spot, the depth-dependent refractive index can be simplified to a piecewise constant function. The corresponding inverse problem has been examined in \cite{BruCha05,ElbMinVes21,TomWan06} and in \cite{ElbMinVes20} where additional inclusions have been discussed. 

In this article, we consider the multi-layer structure of the sample with a Gaussian beam model, presented in \cite{VesKraMinDreElb21}, which resembles more efficiently the laser light illumination. Within this context, we treat the inverse problem of quantifying the refractive index by using a layer-by-layer method, where in each step the reconstruction concentrates on a pair of parameters, the refractive index and the width of the layer. 

Hereby, the reconstruction is formulated as an $\ell^2$-minimization problem in the Fourier domain between the Gaussian beam prediction model and the data. That is, we match the absolute values of both in order to avoid high-frequency components which are typically present. In \cite{ElbMinVes21}, the reconstruction within this setting for the case of a plane wave incident field only gave a unique solution when artificially adding plausible bounds on each parameter. We can omit these bounds in our analysis on the existence and uniqueness of solutions of the present minimization problem.

The analysis itself is complicated by the compact form of the Gaussian model. There the directional dependent reflection coefficients containing information about the refractive indices hinder a deeper investigation. For this reason, the analysis is divided into two parts: Firstly, we consider an approximation of the original model, where the reflection coefficient is assumed to be directional independent. We show that the modeling error caused by this approximation is small and bounded from above. For the approximation we then show the existence and the uniqueness of a solution under the strong but necessary condition that the data is in the range of the prediction. 
The distance in every step is obtained by the biggest overlap of the prediction model with the data once the refractive index has been calculated.  

In order to support our arguments, we show the functionality of the proposed method by reconstructing the refractive indices from both simulated and experimental data. While reconstructions from simulated data, partially under very theortic assumptions, have been discussed several times, the reconstruction using experimental data has been treated fairly rare.   
The assumption on the object to have layered structure is a very strong assumption which turns the inverse problem to the well-posed side. Hence, regularization methods which are typically related to problems with noise, are not considered in this context. 

The outline of this article is as follows: In \autoref{sec:mathmodel} and \autoref{sec:PhysParam} the major parts of the Gaussian beam model introduced in \cite{VesKraMinDreElb21} are summarized. The section ends by giving an explicit representation of a single OCT depth profile. Based on this representation, we introduce the inverse problem which we want to treat layer-by-layer. We discuss this in \autoref{sec:layer}. That the actual reconstruction can be formulated as a layer-by-layer method is justified in \autoref{sec:approximation} and its reformulation as $\ell^2$-minimization problem is shown in \autoref{sec:def_minimization}. Before we present in \autoref{sec:inverse_EaU} a characterization of existence and uniqueness of the solutions to an approximated version of the original problem, we show that the approximation error between the functionals is almost negligible. We conclude the section on the refractive index quantification by presenting a purely theoretical method to also retrieve uniqueness. The analysis on the minimum with respect to the width is presented in \autoref{sec:width} in form of purely visual arguments by showing the behavior of the minimization functional with respect to the width. There, we consider simulated data with and without noise. 
Finally, in \autoref{sec:numerics} the results of the reconstruction from simulated and experimental data of a three-layer object are presented.

\section{OCT Forward Model}
\label{sec:mathmodel}
The image formation in standard OCT systems nowadays is based on raster scanning the object of interest. This means that strongly focused laser light is directed by movable galvanometric mirrors to spots located on an imaginary lateral (horizontal) grid on the object's top surface. For every position then an OCT measurement, a depth profile along the vertical axis through the spot showing the microstructure inside the object, is recorded.   
     
Latest publications in the field of OCT have pointed out that modeling such focused laser light by a single plane wave is not sufficient enough in capturing several system relevant aspects. The huge influence of the focus and the beam width, the diameter of the laser in the focus spot, for example, are a few modeling aspects which are lost when simulating the measurements of an OCT system using single plane waves. For this reason we base our analysis on a Gaussian beam model for the incident field, which is said to model laser light in an accurate way. Hereby, we adapt for our purpose the forward model which has been presented in \cite{VesKraMinDreElb21}. There, a model is provided which includes all relevant system parameters and allows a good approximation of an actual measurement. 
 
In this section we summarize briefly the main parts of an OCT system, namely the light scattering, the fiber coupling and finally the measurement. Hereby, we use the fact that the object is raster scanned, which allows us to model in the following the light propagation for each raster scan independently. At the end of this section we provide an equation for a single raster scan (A-scan) data. This will be used as the data for the corresponding inverse problem. 

The main parts are specifically modeled for a swept-source OCT system \cite{DreFuj08}, since the data used for the numerical experiments in \autoref{sec:numerics} has been obtained by this system type. Within such a system almost monochromatic laser light centered at multiple wavelengths in a certain spectrum is used for the illumination of the object. For each of these wavelengths the scattered light is detected by a narrow-bandwidth interferometer. This finally gives a complete measurement for each wavelength within the spectrum.  

\subsection*{Field of incidence}
We model the light propagation from the point on where the laser light has entered the sample or the reference arm respectively. The presentation restricts to the modeling of the sample arm. The  reference arm, which is modeled analogously with the difference that the object is a perfectly reflecting mirror, is considered as a special case.  

For each raster scan we model the incident illumination as an electromagnetic wave $\hat E:\R\times\R^3\to\C^3,$ a function of the wavenumber $k\in\R$ and the spatial coordinate $x\in\R^3,$ satisfying the system of Maxwell's equations in vacuum, that is (in Fourier space)  
\begin{equation}
\label{eq:helmholtz}
\Delta  \hat E(k,x) + k^2 n_0^2 \hat E(k,x) = 0\quad\text{and}\quad\nabla\cdot \hat E(k,x)=0 \quad \text{for all}\quad  k\in\R,\ x\in\R^3,
\end{equation}
where $n_0 = 1$ is the free-space refractive index. Since we consider a swept-source OCT system, $ \hat E$ is assumed to fulfill the support condition
\[
\supp \hat E(\,\cdot\, ,x) \subset (k_0-\epsilon_0,k_0+\epsilon_0)\quad\text{for all}\quad x\in\R^3,
\]    
for a sufficiently small parameter $\epsilon_0 > 0.$ The experiment is repeated for multiple wavelengths $k_0$ in a spectrum $\mathcal S = [k_1,\, k_2].$
The delta-like support in wavenumber allows us to model each experiment, including the backscattering of the light by the sample, independently. This means we can solve the Helmholtz equation for the electric field for each $k\in\mathcal S$ separately. However, there is no difference to the mathematical formulation for a broadband illumination, an electric field which solves the Helmholtz equation for all $k\in\mathcal S$ simultaneously. Hence, we formulate everything simply for a broadband illumination.

In order to specify the Gaussian beam for the field of incidence, we impose additionally an initial condition at a hyperplane $\{x\in\R^3\ |\ x_3 = r_0\}$, which we locate in the focus of the beam. Hence, by $r_0\in\R$ we denote the focus position, which we always assume to be along the vertical line $\R e_3,$ where $(e_i)_{i=1}^3$ denotes the standard basis of $\R^3.$ For a function $f_k:\R^2\to \C$ with compact support in $D_{k}(0)\subset \R^2,$ the ball with radius $k$ and center zero, and a vector $\eta \in\mathbb S^1\times \{0\},$ we then write
\begin{equation}
\label{eq:incident_plane}
\mathcal F_{\bar x}\left(\hat E\right)(k,\kappa,r_0) = f_k(\kappa) \eta \quad\text{for all}\quad \kappa\in\R^2,
\end{equation}
where $\bar x = (x_1,\,x_2)$ and where 
\[
\mathcal F_{\bar x}(u)(\kappa) = \int_{\R^2} u(\bar x) e^{-i \bar x\cdot\kappa} d \bar x 
\]
denotes the two-dimensional Fourier transform with respect to the variable $\bar x$. 

We consider in the following illumination from the top only. Hence, we eliminate one major propagation direction from the obtained solution $\hat E$ \cite{VesKraMinDreElb21} to the combined problem of \autoref{eq:helmholtz} and \autoref{eq:incident_plane} and consider only downward propagating waves, that is in direction $-e_3.$ We denote this part by $E,$ which for any fixed wavenumber $k\in\mathcal S$ is then represented by       
\begin{equation}
\label{eq:solution}
 E(k,\bar x,x_3) = \frac{1}{4\pi^2} \int_{D_k(0)} g(\kappa) e^{-i \sqrt{k^2-|\kappa|^2}(x_3-r_0)} e^{i \kappa\cdot \bar x} d \kappa, \quad (\bar x,x_3)\in\R^3
\end{equation}
where
\[
g(\kappa) = \frac{1}{2} f_k(\kappa) \left(\eta - \left(\eta_3 - \frac{\kappa\cdot \bar \eta}{\sqrt{k^2-|\kappa|^2}}\right)e_3\right),\quad \bar\eta=(\eta_1,\eta_2).
\]
\begin{remark}
The incident wave in \autoref{eq:solution} represents a weighted superposition of plane waves $e^{i K\cdot x},$ where the wave vector $K$ represents the propagation direction of each plane wave. We note that every wave direction $K$ in \autoref{eq:solution} is implicitely given as a function of $\kappa,$ that is 
\begin{equation}
\label{eq:wavevector}
K(\kappa) = (\kappa,-\sqrt{k^2-|\kappa|^2}).
\end{equation}  
We suppress this dependence in the following.
\end{remark}
We complete the Gaussian beam incident field by specifying the weight function $f_k$ in the focal plane. Hereby, we use Gaussian weights only.  
\begin{assumption}
\label{as:one}
In the following, we restrict our attention to Gaussian distributions, meaning that we consider in \autoref{eq:incident_plane} functions of the form    
\[
f_k(\kappa) = e^{-|\kappa|^2 a}, \quad\text{for all}\quad \kappa\in\R^2,
\] 
where $a>0$ is such that $\|f_k - \bm\chi_{D_{k}(0)}f_k \|_{L^1(\R^2)},$ where $\bm\chi_{\mathcal U}$ is the characteristic function of a set $\mathcal U,$ is almost negligible. 
\end{assumption}
The Gaussian distribution in this case pronounces only those directions $\kappa\in D_k(0)$ which have a small radial deviation from zero, others are excluded from the integral in \autoref{eq:solution}. 
Hence, we can consider the quotient $|\kappa|/k$ as a small parameter for which the paraxial approximation expressed by the linearization of the square-root 
\begin{equation}
\label{eq:paraxial}
\sqrt{k^2-|\kappa|^2} = k - \frac{|\kappa|^2}{2k} + o\left(\frac{|\kappa|^2}{k}\right)
\end{equation}
is valid. We want to include this fact also in the representation of the polarization vector.    
\begin{assumption}
\label{as:two}
Restricting to $\eta = e_2,$ we consider an approximated polarization vector
\[
\tilde g(\kappa) = \frac{1}{2} f_k(\kappa) e_2,\quad \kappa\in\R^2,
\] 
with $f_k$ as in \autoref{as:one} in our model, which reduces the electric field to its transverse components.    
\end{assumption}
We combine \autoref{eq:solution} with \autoref{as:one} and \autoref{as:two} and denote the resulting incident field by $ E^{(0)} = \tilde E^{(0)} e_2,$ with scalar complex field
\begin{equation}
\label{eq:incident_field}
\tilde E^{(0)}(k,x) = \frac{1}{8\pi^2} \int_{D_{k}(0)} e^{-|\kappa|^2 a}e^{i \sqrt{k^2-|\kappa|^2}r_0} e^{i K\cdot x} d \kappa, \quad k\in\R,\ x\in\R^3.
\end{equation}

\subsection*{Light scattering}
The field in \autoref{eq:incident_field} radiates the sample, which we denote by $\Omega.$ We characterize the optical properties of $\Omega$ by the refractive index represented by a function $n:\Omega\subset\R^3\to [1,\infty).$ The assumption that $n$ is a real function refers to the case of a sample where the scattering dominates over the thus neglected absorption. Additionally, because of the small bandwidth of the laser, we assume that $n$ is constant with respect to the wavenumber in the spectrum $\mathcal S.$

The main assumption however is that $\Omega$ shows a multi-layered structure, meaning that $\Omega$ consists of a finite union of subsets $\Omega_j,$ where each is characterized by a homogeneous refractive index $n_j\geq 1,$ with $n_j\neq n_{l},\, l\in\{j-1,\, j+1\}.$ We claim that all layers are parallel and therefore share a single unit normal vector $\nu_\Omega = (\sin\theta_\Omega,0,\cos\theta_\Omega),$ for a small angular value of $\theta_\Omega\in\R,$ which we assume for simplicity to be orthogonal to the polarization vector $\eta=e_2$ and which is pointing outward the object. 

To summarize, we consider a sample given by 
\begin{equation}
\label{eq:sample}
\Omega = \bigcup_{j=1}^J \Omega_j,\quad \Omega_j = \left\{x\in\R^3\ |\ a_{j+1}\leq \left( x\cdot \nu_\Omega\right)\leq a_{j}\right\},\quad n(x) = \sum_{j=1}^J \bm \chi_{\Omega_j}(x) n_j,
\end{equation}    
for a sequence of coefficients $(a_j)_j\subset \R.$ The total number of layers $J\in\N$ hereby is an unknown but arbitrary finite number. We additionally assign to every layer $\Omega_j$ its width, which is given by the positive real number $d_j = a_j - a_{j+1}.$ 
The light $E:\R\times\R^3\to\C^3$  in presence of the sample, \autoref{eq:sample}, is then modeled as a solution of the vectorial Helmholtz equation for all $x\in\R^3$:
\[
\nabla\times(\nabla \times E)(k,x) - k^2 \tilde{n}(x)^2 E(k,x) = 0 \quad \text{with}\ \ \tilde{n}(x)= \begin{cases} n_0, & x\in\R^3\setminus \Omega,\\ n(x), & x\in\Omega. \end{cases}
\]
By taking the divergence of this equation, we see that the assumption of piecewise homogeneous layers implies that $\nabla\cdot E=0$ in every set $\Omega_j$, so that $\nabla\times(\nabla \times E)=-\Delta E$ and the equation reduces to the simpler Helmholtz equation (see \autoref{eq:helmholtz}) for the second component $E_j$ (the first and third components vanish by our choice of incident field) of the field $E$ inside the layer $\Omega_j$, similar to \cite{Fer96,MarRalBopCar07}, where $k^2n_0^2$ is replaced by $k^2 \tilde{n}^2.$ At the boundaries between the single layers, we claim that the homogeneous parts of the electric field fulfill the continuity conditions
\begin{equation*}
E_j(k,x) = E_{j+1}(k,x),\quad \nabla E_j(k,x)\cdot \nu_\Omega = \nabla E_{j+1}(k,x) \cdot \nu_\Omega, 
\end{equation*}
for all $x\in\partial \Omega_j \cap \partial \Omega_{j+1}$ and $j\in \{1,\dots,J\}.$ We finally denote by $E^{(s)}$ the second component of the backscattered field $E - E^{(0)}$ from the object. 

Since the Fresnel equations allow us to get an explicit solution for the case of one discontinuity, we can iteratively construct a solution for a finite number of layers as a series of Fresnel solutions, which physically corresponds to multiple light reflections at the boundaries, see \cite{BruCha05,ElbMinVes21}.
Since multiply relfected light only contributes minorly, we ignore it in the model for the backreflected sample field. However, we keep in mind that these multiple reflections, at least up to a finite order, can be added to the model at any time. The assumption of a single reflection model makes it possible to use the transmitted part of the light at a certain layer directly as an incident field to the next layer.   

A layer-by-layer scheme, similar to the one in \cite{ElbMinVes21}, allows us to give an explicit representation of the backscattered field $E^{(s)}.$ The incident field hereby is decomposed into its plane wave parts. For each, we use the Fresnel formulas, see \cite[Chapter~7]{Jac98}, to determine the reflection coefficients (as functions of the propagation direction of the incoming plane wave) which we denote by $r_j:\R^2\to \R$ for every interface $\partial\Omega_{j-1}\cap\partial\Omega_j$ (with $\Omega_0=\Omega_{J+1}=\R^3\setminus\Omega$). 
Finally, the reflected fields are combined again and we obtain for $x,$ with $x_3 > x_{\Omega,3},$ where $x_\Omega$ is a point on the top surface $\{x_\Omega\in\R^3:x_\Omega\cdot\nu_\Omega=a_1\},$ the backscattered field  
\begin{multline}
\label{eq:scatfield}
E^{(s)}(k,x) = \frac{1}{8\pi^2} \sum_{j=1}^{J+1} \int_{\R^2} r_j(\kappa) \left(r_{\leq j-1}(\kappa) e^{i k \Psi_j(\kappa)}\right) e^{-|\kappa|^2 a} \\
\times e^{i \sqrt{k^2-|\kappa|^2}r_0} e^{i (K-\Phi(K))\cdot x_\Omega } e^{i \Phi(K)\cdot x} d \kappa,
\end{multline}
with 
\begin{equation}
\label{eq:reflection_coe}
r_j(\kappa) = \frac{n_{j-1} \cos\theta^{j-1}_t(\kappa) - \sqrt{n_j^2-n_{j-1}^2+n_{j-1}^2 \cos^2\theta^{j-1}_t(\kappa)}}{n_{j-1}\cos\theta^{j-1}_t(\kappa) + \sqrt{n_j^2-n_{j-1}^2+n_{j-1}^2 \cos^2\theta^{j-1}_t(\kappa)}}
\end{equation}
and where we denote by $K=K(\kappa)=(\kappa,-\sqrt{k^2-|\kappa|^2})$ the wave vector (introduced in \autoref{eq:wavevector}) and by $\Phi(K)=K - 2\left( K\cdot\nu_\Omega\right)\nu_\Omega$ the wave direction of the reflection of a plane wave with incident vector $K$ at an interface with unit normal vector $\nu_\Omega$. Moreover, we introduced for $j\geq 1$ the transmission coefficients and the (transmission) phase factors
\begin{equation}
\label{eq:phase}
r_{\leq j-1}(\kappa) = \prod_{l=1}^{j-1} (1-r_l^2(\kappa)), \quad\quad  \Psi_{j}(\kappa) = 2\sum_{l=1}^{j-1} n_l d_l \cos\theta^{l}_t(\kappa),
\end{equation}
where the angles $\theta^{j}_t(\kappa)$ of transmission between the layers $j$ and $j+1$ for an incident plane wave with wave vector $K(\kappa)$ can be iteratively calculated via Snell's law
\begin{equation}
\label{eq:angleoft}
\theta^{j}_t(\kappa) = \arcsin \left( \frac{n_j}{n_{j+1}}\sin\big(\theta^{j-1}_t(\kappa)\big) \right) \quad\text{with}\quad \theta^{0}_t(\kappa) = \arccos\left(-\tfrac{K}{|K|}\cdot\nu_\Omega\right).
\end{equation}
\begin{remark}
\label{eq:dimless}
We remark that since $|K(\kappa)| = k,$ the quotient 
\[
\frac{K(\kappa)}{|K(\kappa)|} = \left(\frac{\kappa}{k},-\sqrt{1-\frac{|\kappa|^2}{k^2}}\right)
\] 
in the definition of the angle of incidence $\theta_t^0$ (in \autoref{eq:angleoft}) is actually a function of the dimensionless parameter $\frac{\kappa}{k}.$ This implies that the same is true also for the angle of transmission $\theta_t^j$ and the phase factor $\Psi_j.$ Hence, we will always write $\Psi_j(\frac{\kappa}{k})$ in the following. 
\end{remark}

\subsection*{Fiber coupling}
The backscattered light is then coupled into a single-mode fiber and later transfered to the detector. This coupling is done by using a scan lens which discards all (plane wave) parts in \autoref{eq:scatfield} of which wave direction strongly deviates from the third unit normal $e_3.$ The maximal deviation is described by the (maximal) angle of acceptance, which we denote by $\theta.$ The set of incident directions yielding an accepted wave vector is then given by 
\begin{equation}
\label{eq:accepwavevec}
\mathcal B = \left\{\kappa\in\R^2\ \Big|\ \arccos\left( \frac{1}{k} \left(\Phi(K)\cdot e_3 \right) \right) \leq \theta  \right\}\subset \R^2. 
\end{equation}
This means that the area of integration in \autoref{eq:scatfield} is reduced to the set $\mathcal B\subset\R^2.$ 

Because of the restriction to this for typical cases small set the deviation in the directions $\kappa$ is limited and small. Hence, also the reflection coefficients are only varying slighty (with $\kappa$) which motivates us to consider an approximated form of \autoref{eq:scatfield}, especially in context of the inverse problem in \autoref{sec:approximation} and \autoref{sec:def_minimization}, where the reflection coefficients are directional independent.\\
For example, if $\theta_\Omega = 0,$ the accepted directions are given by $\mathcal B = D_{k \sin\theta}(0),$ the ball centered at zero with radius $k\sin\theta.$ The inclination angle of $K(\kappa)$ with $\nu_\Omega$ is maximal if $|\kappa| = k \sin \theta,$ which means that $\kappa$ is on the boundary of $\mathcal B.$ The corresponding reflection coefficient is given by
\begin{equation*}
r(\kappa) = \frac{n \cos\theta^0_t(\kappa) - \sqrt{(n')^2-n^2\sin^2(\theta_t^0(\kappa))}}{n \cos\theta^0_t(\kappa) + \sqrt{(n')^2-n^2\sin^2(\theta_t^0(\kappa))}} = \tilde r\left(\theta^0_t(\kappa)\right)
\end{equation*}
(for a single boundary reflection between $n$ and $n'$), which we then can write as a function of the inclination angle $\theta^0_t.$ For $|\kappa| = k\sin\theta,$ we get $\theta_t^0=\theta$ and we can use the approximation
\[
\tilde r(\theta) \approx \frac{n \cos\theta - n' + \tfrac{n^2}{n'}(1-\cos^2\theta)}{n \cos\theta + n' - \tfrac{n^2}{n'}(1-\cos^2\theta)}.
\]
Under the assumption of $\cos\theta\approx 1,$ which is a good approximation in our setting as $\theta$ is around $2^\circ,$ the reflection coefficient is approximated by the constant $\frac{n-n'}{n+n'},$ the classical Fresnel coefficient \cite{Jac98}.  
  
In general, the angular values of $\theta$ and $\theta_\Omega$ are considered to be small, which allows us to assume that the variation of the reflection coefficient on $\mathcal B$ is also well approximated by a constant coefficient. Next, we show that for small angular values the difference between the reflection coefficient (as a function of the angle of incidence), see \autoref{eq:reflection_coe}, and its approximating constant tends to zero quadratically. 
  
\begin{lemma}
\label{lem:ref_approx}
Let $n,n'\in\R$ and let $r^\dagger = \displaystyle\frac{n-n'}{n+n'}\in\R.$ Further, let $r:[-\pi/2,\pi/2]\to\R$ be defined by 
\[
r(y) = \frac{n\cos(y) - \sqrt{n'^2-n^2 + n^2 \cos^2(y)}}{n\cos(y) + \sqrt{n'^2-n^2 + n^2 \cos^2(y)}}.
\] 
Then for small values of $y$ we have that 
\[
\lim_{y\to 0} \left|\frac{r^\dagger- r(y)}{y}\right| = 0.
\]
\end{lemma}

\begin{proof}
We determine the Taylor expansion of $r$ locally around the point $y_0 = 0.$ Obviously, $r(y) = r(-y),$ for all $y\in\R,$ which makes $r$ an even function. Thus, for its Taylor expansion we expect only even exponents. 

Estimating now the difference between $r^\dagger$ and $r$ for small values of $y$ gives 
\[
|r^\dagger- r(y)| \leq \frac{|n^3-2n n'^2|}{2|n'|(n+n')^2} y^2,  
\]
which proves the result.
\end{proof}
Applying this formula in our context, we obtain the strongest deviation from $r^\dagger$ on the (closed) set $\mathcal B$ in the boundary point $\kappa$ with maximal angular deviation $\theta_0= \theta + 2 \theta_\Omega$ from the unit vector $-\nu_\Omega.$ Thus, we obtain a maximal error when approximating $r$ by $r^\dagger$ by the order of $(\theta +2 \theta_\Omega)^2.$ We silently assume that the refraction at the deeper interfaces does not change the angle too much so that this approximation remains valid at all interfaces.

\subsection*{OCT measurement}
At the detector, an interference pattern of the backscattered field from the sample combined with a reference field is measured. This reference field is obtained by the reflection of the incident light by a perfectly reflecting non-tilted mirror, which is modeled as semi-infinite object with (infinitely) large refractive index. We assume that the mirror is located around the focus of the beam where the far-field approximation \cite{Hoe03, ManWol95} is considered to be valid \cite{VesKraMinDreElb21}.  
For the (measurement) direction $s = e_3$ and the distance $\rho\in \R,\ 1\ll \rho,$ we then obtain an interference pattern 
\[
\mathcal C(k)=\Re(E^{(s)}(k,\rho e_3)\overline{E_\infty^{(0)}(k,\rho e_3)})
\] 
where $E_\infty^{(0)}(k,\rho e_3)$ denotes the far field approximation of the reflected incident field. Using the representation of the backscattered field, we obtain
\begin{equation}
\label{eq:interfer}
\mathcal C(k) =  -\frac{k}{16\pi^3 \rho} \sum_{j=1}^{J+1} \int_{\mathcal B} \left(r_j(\kappa) r_{\leq j-1}(\kappa)\right) e^{-|\kappa|^2 a}
 \sin\left(k\left(-|\kappa|^2 \frac{\psi_0}{2k^2} + \frac{\kappa_1}{k}\psi_1 + \Delta_0 + \Psi_{j}\left(\tfrac{\kappa}{k}\right)\right)\right)d \kappa,
\end{equation}
where we used the paraxial approximation, see \autoref{eq:paraxial}. The elements 
\begin{equation}
\label{eq:psi0}
\psi_0 = r_0 - \rho - 2 \cos^2(\theta_\Omega) (x_{\Omega,3}-\rho),\quad \psi_1 = \sin\left(2\theta_\Omega\right)(x_{\Omega,3}-\rho)
\end{equation}
and $\Delta_0$ describe the position of the object with respect to the focus and the phase difference between the sample and the reference mirror respectively, see \cite{VesKraMinDreElb21}. 

The backscattered light is recorded for each wavenumber $k\in\mathcal S$ independently. Collecting these single measurements, leads to the measurement $\mathcal C$ as a function of the wavenumber.
The interference pattern in \autoref{eq:interfer} hereby represents the OCT measurement comprising one A-scan of the object, which is the center of interest in the OCT forward and inverse problem.

We want to assume now that all the parameters of the system are known and that the unknown is the refractive index function $n$ of the medium. 
With the forward operator $\mathcal I\colon\R^{J+1}\times\R^J\to L^2(\mathcal S)$ 
\begin{multline}
\label{eq:forward_integral}
\mathcal I[(n_j)_{j=1}^{J+1},(d_j)_{j=1}^J](k) = \frac{-k}{16\pi^3 \rho} \sum_{j=1}^{J+1} \int_{\mathcal B} \left(r_j(\kappa) r_{\leq j-1}(\kappa)\right) e^{-|\kappa|^2 a}\\\times \sin\left(k \left(-|\kappa|^2\frac{\psi_0}{2k^2} + \frac{\kappa_1}{k}\psi_1 +\Delta_0 + \Psi_{j}\left(\tfrac{\kappa}{k}\right)\right)\right)d \kappa,
\end{multline}
we can thus write the inverse problem of OCT as the determination of $(n_j)_{j=1}^{J+1}$ and $(d_j)_{j=1}^J$ from the measurements $\mathcal C$ via the non-linear equation
\begin{equation}
\label{eq:inverse_problem}
\mathcal I[(n_j)_{j=1}^{J+1},(d_j)_{j=1}^J](k) = \mathcal C(k),\quad k\in\mathcal S. 
\end{equation}

\section{The Physical Parameters}
\label{sec:PhysParam}
So far \autoref{eq:inverse_problem} together with \autoref{eq:forward_integral} has been formulated as a general inverse scattering problem for interferometric measurement data. We now want to simplify the equations by plugging in the typical values of the parameters in an OCT system and keeping only terms of significant size.
In \autoref{tab:parameters}, we listed the working parameters of the OCT system used for our data acquisition. These are kept fixed for the whole modeling process and also the inverse problem.

\begin{table}[hbt!]
\centering
\def\arraystretch{1.25}
\begin{tabular}{ |l|c|c| } 
\hline
Parameter & Symbol & Value \\
\hline
 \hline
spectrum of wavenumbers & $\mathcal S=[k_1,k_2]$ & [\SI{4.78}{\per\micro\metre},\SI{4.9}{\per\micro\metre}] \\
mean of the spectrum & $\bk k=\frac12(k_1+k_2)$ & \SI{4.84}{\per\micro\metre} \\
half bandwidth of the spectrum & $\lk k=\frac12(k_2-k_1)$ & \SI{0.057}{\per\micro\metre} \\
beam width & $2\sqrt{a}$ & \SI{15}{\micro\metre} \\ 
angle of acceptance & $\theta$ & \SI{2.08}{\degree} \\
tilting angle of the object & $\theta_\Omega$ & \SI{1.20}{\degree} \\
distance to the detector & $\rho$ & \SI{63000}{\micro\metre}\\
\hline
\end{tabular}
\caption{Parameters of the experimental setup of the considered OCT system.}
\label{tab:parameters}
\end{table}

\begin{table}[hbt!]
\centering
\def\arraystretch{1.25}
\begin{tabular}{ |c|c| } 
\hline
Parameter & Value \\
\hline
 \hline
$\lk k\sqrt{\gamma}=\lk k\sqrt{a}\sin\theta$ & \num{1.6e-2} \\ 
$\lk k\bk k^{-1}$ & \num{1.1e-2} \\
$\bk k a\psi_0^{-1}$ & \num{4.6e-3}\\
$(\lk k\Delta_0)^{-1}$ & \num{3.9e-3} \\
$\theta^2$ & \num{1.3e-3} \\
$\theta_\Omega^2$ & \num{4.3e-4} \\
\hline
\end{tabular}
\caption{Small quantities in the experimental setup.}
\label{tab:small_parameters}
\end{table}

In particular, we have a rather narrow bandwidth, a small beam width, an almost normal incident angle, a small angle of acceptance, and we measure at a comparatively large distance to the object. This leads to some combinations of and relations between these parameters, which are small compared to one and which we will often simply neglect in the following analysis.

\begin{assumption}
\label{as:three}
We assume that the following quantities can be considered sufficiently small so that they can be safely neglected in the model.

\begin{enumerate}
\item In general, we assume the angular values $\theta$ and $\theta_\Omega$ to differ only slightly from zero, which allows us to consider both quantities as small parameters and to keep in the following only terms which are linear with respect to them. 
\item
We assume that the position of the detector is sufficiently far away from the object, which implies that in the relation in \autoref{eq:psi0} the distance between the focus and the object is dominated by the distance between the object and the detector, yielding that $\psi_0 \approx \rho-x_{\Omega,3},$ with $\rho \gg x_{\Omega,3}.$ In particular, we assume that the ratio $\frac{\bk k a}{\psi_0}$ between half the beam width $\bk k\sqrt a$ (measured in multiples of the averaged wave length) and the distance $\frac{\psi_0}{\sqrt a}$ to the detector (in multiples of half the beam width) is neglible.
\item
Similarly, we assume that the distance $\Delta_0$ between the object in the sample arm and the mirror in the reference arm is so large that the difference $k_2\Delta_0-k_1\Delta_0$ in multiples of the wave lengths between measuring it with the maximal wave vector $k_2$ and the minimal $k_1$ is large, so that we can assume $(\lk k\Delta_0)^{-1}$ to be small.
\item
The ratio between the bandwidth $2\lk k$ of the spectrum and its center $\bk k$ is assumed to be so small that it is enough to keep the zeroth order in $\lk k\bk k^{-1}$.
\item
Finally, we assume that the beam width is so small that the deviation $\lk k\sqrt{a}\sin(\theta)$ of the tilt measured with respect to the different wave lengths in the specturm is sufficiently small to neglect all terms of higher than linear order therein.
\end{enumerate}

We summarize all these small quantities and their values in our experimental setup in \autoref{tab:small_parameters}.
\end{assumption} 
  
\section{A Layer-by-layer Method for the Inverse Problem}
\label{sec:layer}
To avoid having to solve for all the $2J+1$ parameters in \autoref{eq:inverse_problem} at once, we want to try to split the reconstruction, as it was also done in \cite{ElbMinVes21, Som94, SylWinGyl96}, for example, into the subproblems
\begin{equation}\label{eq:inverse_problem_j}
\mathcal I_j[n_j,d_{j-1}](k) = \mathcal C^j(k),\quad\text{for}\quad j=1,\ldots,J,
\end{equation}
where $\mathcal I_j\colon\R\times\R\to L^2(\mathcal S)$ is the $j$th term in the sum of \autoref{eq:forward_integral}, which corresponds to the contribution from the reflection at the boundary between the layers $\Omega_{j-1}$ and $\Omega_j$:
\begin{multline}
\label{eq:integral}
\mathcal I_j[n_j,d_{j-1}](k) = \frac{-k}{16\pi^3 \rho} \int_{\mathcal B} r_j(\kappa) r_{\leq j-1}(\kappa)\\
 \times e^{-|\kappa|^2 a}\sin\left(k \left(-|\kappa|^2\frac{\psi_0}{2k^2} + \frac{\kappa_1}{k}\psi_1 +\Delta_0 + \Psi_{j}\left(\tfrac{\kappa}{k}\right)\right)\right)d \kappa.
\end{multline}
Since also the coefficients $(n_l)_{l=1}^{j-1}$ and $(d_l)_{l=1}^{j-2}$ from the previous interfaces appear in $\mathcal I_j$ via the combined reflection coefficients $r_{\leq j-1}$ and the combined phase factor $\Psi_j$ of the transmissions at the previous interfaces, we want to proceed iteratively by recovering first $n_1$ from \autoref{eq:inverse_problem_j} for $j=1$, and then obtain $(n_j,d_{j-1})$ from the $j$th problem in \autoref{eq:inverse_problem_j} after having already recovered $(n_l)_{l=1}^{j-1}$ and $(d_l)_{l=1}^{j-2}$ from the previous ones.

The difficulty hereby is, however, that we a priori do not have access to the corresponding measurements $\mathcal C^j$. To get an approximation for these, we perform a Fourier transform of \autoref{eq:inverse_problem} with respect to the wave number $k$ (we extend the function from $\mathcal S$ to $\R$ by zero) 
\[
\mathcal F_k(u)(z) = \frac{1}{\sqrt{2\pi}}\int_\R u(k) e^{-ikz} dk
\]
and obtain
\begin{equation}
\label{eq:invp_fft}
\lk k\sqrt{\frac{2}{\pi}} \sum_{j=1}^{J+1} \left(\sincs(\lk k\,.\,)e^{-i \bk k\,.\,}\right)*_z\mathcal F_k(\mathcal I_j) = \lk k\sqrt{\frac{2}{\pi}} \left(\sincs(\lk k\,.\,)e^{-i \bk k\,.\,}\right)*_z\mathcal F_k(\mathcal C),
\end{equation}
with $\lk k$ and $\bk k$ as in \autoref{tab:parameters} and where we call the dual variable to the wave number $k$ the optical distance $z$ and write $*_z$ for the convolution with respect to $z$, see \autoref{fig:data2}. This is a combination of sinc-functions (defined by $\sincs(z) = \sin(z)/z$), which are centered at the frequencies 
\[
\tilde\Delta_j(\tilde \kappa) = -|\tilde \kappa|^2\frac{\psi_0}{2} + \tilde \kappa_1\psi_1 + \Delta_0 +\Psi_j(\tilde \kappa),\ \tilde \kappa = \frac{\kappa}{k},\quad j=1,\dots,J,\ \kappa\in\mathcal B,
\] 
of the sine under the integral in \autoref{eq:interfer}. Since $\mathcal B$ is a small disk close to the origin, we get in a very rough approximation
\[ \tilde \Delta_j(\tilde\kappa)-\tilde\Delta_{j-1}(\tilde\kappa) \approx \tilde \Delta_j(0)-\tilde\Delta_{j-1}(0) = 2n_{j-1}d_{j-1}, \]
so that the distance between the peaks of the sinc-functions corresponds in zeroth order to twice the travel time of the light between the two interfaces.

If in addition to the width of the layers, the size of the spectrum $\mathcal S,$ described by $\lk k$, is sufficiently large (which is the basis of OCT), these peaks can be nicely separated from each other so that we find intervals $\mathcal U_j$ around the points $\tilde\Delta_j(\frac\kappa k)$, $\kappa\in\mathcal B$, so that
\[ \lk k\sqrt{\frac{2}{\pi}} \sum_{l=1}^{J+1} \left(\sincs(\lk k\,.\,)e^{-i \bk k\,.\,}\right)*_z\mathcal F_k(\mathcal I_l) \approx \lk k\sqrt{\frac{2}{\pi}} \left(\sincs(\lk k\,.\,)e^{-i \bk k\,.\,}\right)*_z\mathcal F_k(\mathcal I_j)\quad\text{on}\ \mathcal U_j. \]

We will therefore assume that we can recast the inverse problem from \autoref{eq:inverse_problem} as the iterative procedure, where we start from the top and then go layer by layer deeper inside the object and recover in the $j$th step from the knowledge of $(n_l)_{l=1}^{j-1},(d_l)_{l=1}^{j-2}$ the parameters $(n_j,d_{j-1})$ from
\begin{equation}
\label{eq:inv_layer}
\lk k\sqrt{\frac{2}{\pi}} \left(\sincs(\lk k .)e^{-i \bk k .}\right)*_z\mathcal F_k(\mathcal I_j[n_j,d_{j-1}]) = \lk k\sqrt{\frac{2}{\pi}} \left(\sincs(\lk k .)e^{-i \bk k .}\right)*_z\mathcal F_k\left(\mathcal C^j\right), \quad z\in \mathcal U_j,
\end{equation}
where $\mathcal I_j:\R^2\to L^2(\mathcal S)$ is the forward operator for the $j$th light-layer interface interaction, defined by \autoref{eq:inverse_problem_j} and $\mathcal U_j$ is a suitably chosen interval around the peaks of the data $\mathcal C^j = \mathcal C - \sum_{l=1}^{j-1} \mathcal I_l[n_l,d_{l-1}]$.

We note that the operator $\mathcal I_1$, corresponding to the reflection at the top boundary, only takes the refractive index as an argument and no distance.

\begin{figure}
\centering
\includegraphics[scale = 1.2]{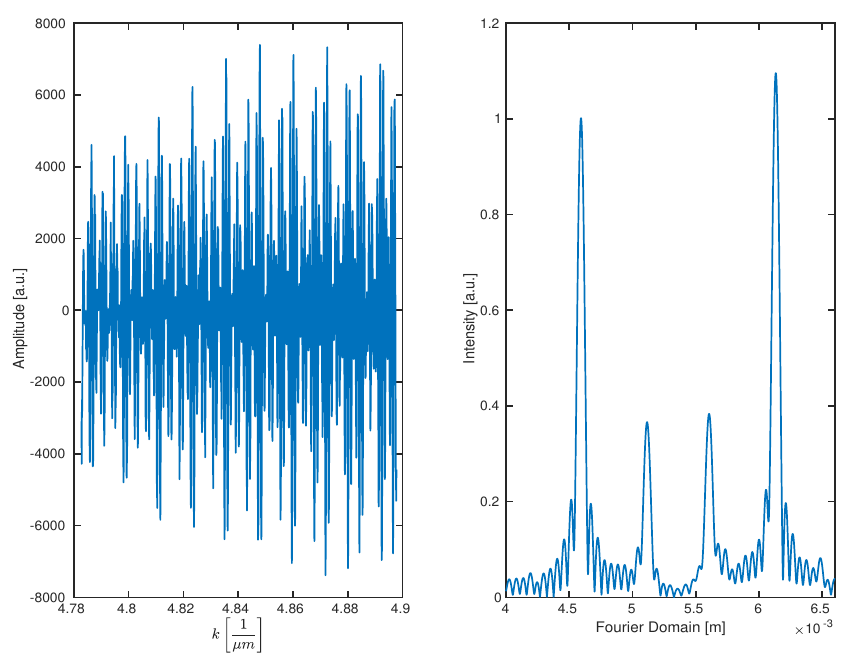}
\vspace{-1em}
\caption{A single A-scan in wavenumber domain showing high-frequent components (left). The absolute value of the dataset after Fourier transform avoiding these high-frequency components in the data (right).}
\label{fig:data2}
\end{figure}

\section{Almost Normal Incidence}
\label{sec:approximation}
We will thus look at a specific step $j$ of the inverse problem \autoref{eq:inv_layer}. To discuss its properties, we want to approximate in $\mathcal I_j$ the refractive index $r_j$ by the expression $r_j^\dag=\frac{n_{j-1}-n_j}{n_{j-1}+ n_j}$, which corresponds to the refractive index a plane wave with normal incidence on the interface would experience. This yields the simplified forward operator
\begin{equation}
\label{eq:integral_approx}
\mathcal I^*_j[n_j,d_{j-1}](k) = \frac{-kr^\dagger_j}{16\pi^3 \rho}\int_{\mathcal B} r_{\leq j-1}(\kappa)e^{-|\kappa|^2 a}\sin\left(k \left(-|\kappa|^2\frac{\psi_0}{2k^2} + \frac{\kappa_1}{k}\psi_1 +\Delta_0 + \Psi_{j}\left(\tfrac{\kappa}{k}\right)\right)\right)d \kappa.
\end{equation}

Since the angle $\theta_\Omega$ is assumed to be small, we can expand this around the value $\theta_\Omega=0$, at which we have $\mathcal B=D_{k\sin\theta}(0)$ and $\psi_1=0$, and get an analytic expression for the leading order term of the forward operator. We will do the calculations for the operator $\mathcal I^*_1$ corresponding to the first interface, which is a bit simpler than those for deeper interfaces, since also the terms $r_{\leq 0}=1$ and $\Psi_1=0$, defined in \autoref{eq:phase}, disappear.

\begin{lemma}
\label{lem:zero_order}
Let $\theta_\Omega = 0$ and let $\mathcal I^*_1$ be defined as in \autoref{eq:integral_approx}. Then, we have that 
\begin{equation}
\label{eq:zero_order}
\mathcal I^*_1[n_1](k) = i \left(\frac{r_1^\dagger k^2}{16\pi^2\rho} \right)\sum_{\epsilon\in\{-1,1\}}\epsilon\frac{e^{i\epsilon k\Delta_0}}{2 a k + i\epsilon \psi_0}\left( 1- e^{-k^2\gamma} e^{  -i\epsilon k  \xi}\right)
\end{equation}
with 
\begin{equation}
\label{eq:new_parameters}
\gamma = a\sin^2(\theta)\quad\text{and}\quad\xi = \frac{\psi_0}{2} \sin^2(\theta).
\end{equation}
\end{lemma}
\begin{proof}
For $\theta_\Omega = 0,$ the set of accepted wave vectors $\mathcal B$ is defined by the ball with radius $k\sin\theta$ and center zero. By using
\[
\sin(x) = \frac{1}{2i} \left(e^{ix} - e^{-ix}\right)
\]
we rewrite \autoref{eq:integral_approx} as 
\begin{align*}
\mathcal I^*_1[n_1] &= i \left(\frac{r_1^\dagger}{16\pi^3\rho}\right) \frac{k}{2} \int_{D_{k\sin\theta}(0)} e^{-|\kappa|^2 a} \left( e^{-|\kappa|^2\frac{i}{2k} \psi_0}e^{i k \Delta_0} -  e^{|\kappa|^2 \frac{i}{2k} \psi_0}e^{-i k \Delta_0}\right) d\kappa \\ 
&= i \left(\frac{r_1^\dagger}{16\pi^3\rho}\right) \frac{k}{2}\sum_{\epsilon\in\{-1,1\}}\epsilon e^{i\epsilon k \Delta_0}\int_{D_{k\sin\theta}(0)} e^{-|\kappa|^2(a+\epsilon\frac{i}{2k} \psi_0)}d\kappa. 
\end{align*} 
By switching to polar coordinates for $\kappa$, we can explicitly calculate this integral and find the desired representation.
\end{proof}

Taking the Fourier transform of \autoref{eq:zero_order} with respect to the wavenumber $k$, we obtain in leading order a representation of $\mathcal F_k(\mathcal I^*_1[n_1])$ as a sum of sinc-functions.
\begin{lemma}
\label{lem:sinc_representation}
Let $\theta_\Omega = 0$ and for $k\in\mathcal S=[k_1,\, k_2]$ let $\mathcal I^*_1$ be defined as in \autoref{eq:integral_approx}. Then its bandlimited Fourier transform is given as the linear combination 
\begin{align}
\label{eq:sinc_app}
\mathcal F_{k}(\mathcal I^*_1[n_1]&\bm\chi_\mathcal S)(z) = -\frac1{8\pi^2\sqrt{2\pi}}\frac{\lk k \bk k^2 r_1^\dagger}{\rho\psi_0}\sum_{\epsilon\in\{-1,1\}}e^{-i\bk k(z-\epsilon\Delta_0)}\Bigg[u_2(z-\epsilon\Delta_0) - \frac{2a\bk k\epsilon}{\psi_0} u_3(z-\epsilon\Delta_0)\\
&- e^{-\bar k^2 \gamma}e^{-i \bar k \epsilon\xi}\left( (u_2(z-\epsilon(\Delta_0-\xi)) +\mathcal O(\lk k\sqrt\gamma))- \frac{2a\bk k\epsilon}{\psi_0}(u_3(z-\epsilon(\Delta_0-\xi)) +\mathcal O(\lk k\sqrt\gamma)) \right)\Bigg].\nonumber
\end{align}
where the functions $u_j:\R\to \C$ are defined by 
\begin{equation}
\label{eq:sinc_der}
u_j(z) = \bk k^{-j}\frac{d^j}{d z^j}\left(\sincs(\lk k z)e^{-i \bar{k} z}\right)e^{i \bar{k} z}
\end{equation}
and we define the parameters $\gamma$ and $\xi$ as in \autoref{eq:new_parameters}.    
\end{lemma}

\begin{proof}
The Fourier transform of \autoref{eq:zero_order} yields  
\begin{equation}
\label{eq:Fourier_integral}
\mathcal F_{k}(\mathcal I^*_1[n_1]\bm\chi_\mathcal S)(z) =
 i\left(\frac{r_1^\dagger}{16\pi^3\rho}\right)\sqrt{\frac{\pi}{2}}\sum_{\epsilon\in\{-1,1\}}\epsilon \int_{k_1}^{k_2} \frac{k^2(2 a k- i\epsilon \psi_0)}{4 a^2 k^2 + \psi_0^2}e^{-i k (z-\epsilon\Delta_0)}  \left( 1 - e^{-k^2\gamma}  e^{ -i\epsilon k \xi }  \right) dk
\end{equation}
Because of \autoref{as:three} we have that $ak$ is small compared to $\psi_0$ and we linearize the fraction
\[
\frac{2ak - i\epsilon \psi_0}{4a^2k^2 +\psi_0^2} = \frac{2ak - i\epsilon\psi_0}{\psi_0^2}\left(1+\mathcal O\left(\frac{a^2k^2}{\psi_0^2}\right)\right).
\] 
We split this into the linear and the constant term in $k$ and see that the integrals in \autoref{eq:Fourier_integral} reduce to the two types
\begin{align*}
&I_{j,\epsilon}^{(1)}(z) = \int_{k_1}^{k_2} k^j e^{-i k (z-\epsilon\Delta_0)} dk \quad\text{and} \\
&I_{j,\epsilon}^{(2)}(z) = \int_{k_1}^{k_2} k^j e^{-i k (z-\epsilon\Delta_0)} e^{-k^2\gamma}  e^{ -i\epsilon k \xi }dk
\end{align*}
of integrals for $j\in\{2,3\}$ and $\epsilon\in\{-1,1\}$.

\begin{itemize}
\item
We calculate the first type by writing $k^j$ as derivative with respect to the variable $z$ and obtain with $\bk k=\frac12(k_1+k_2)$ and $\lk k=\frac12(k_2-k_1)$ that
\[ I_{j,\epsilon}^{(1)}(z) = i^j\int_{k_1}^{k_2} \partial_z^j\left(e^{-i k (z-\epsilon\Delta_0)}\right) dk = 2i^j\lk k\,\partial_z^j\big(\sincs(\lk k(z-\epsilon\Delta_0))e^{-i\bk k(z-\epsilon\Delta_0)}\big). \]
\item
For the second integral, we proceed analogously which gives us together with the Fourier convolution theorem
\begin{align*}
I_{j,\epsilon}^{(2)}(z) &= -i^{j-2}\int_{k_1}^{k_2} \partial_z^{j}\big(e^{-i k (z-\epsilon(\Delta_0-\xi))}\big)e^{-k^2\gamma}dk \\
&= -i^{j-2}\frac{\lk k}{\sqrt{\pi\gamma}}\int_{\R}e^{-\frac{\zeta^2}{4\gamma}}\partial_z^{j}\left(\sincs(\lk k(z-\zeta-\epsilon(\Delta_0-\xi)))e^{-i\bk k(z-\zeta-\epsilon(\Delta_0-\xi))}\right) d\zeta.
\end{align*}
We substitute $\tilde\zeta=\frac\zeta{\sqrt{\gamma}}$ to rewrite this in the form
\[ I_{j,\epsilon}^{(2)}(z) = -i^{j-2}\frac{\lk k}{\sqrt\pi}\int_{\R}e^{-\frac14\tilde\zeta^2}\partial_z^{j}\left(\sincs(\lk k(z-\tilde\zeta\sqrt\gamma-\epsilon(\Delta_0-\xi)))e^{-i\bk k(z-\tilde\zeta\sqrt\gamma-\epsilon(\Delta_0-\xi))}\right) d\tilde\zeta. \]
Using that, according to \autoref{as:three}, the term $\lk k\sqrt\gamma$ is much smaller than one and the integral is due to the Gaussian factor $e^{-\frac14\tilde\zeta^2}$ restricted to a domain of order one, we approximate the sinc-function in the integral by its zeroth order in terms of $\tilde \zeta \lk k\sqrt\gamma.$ The mean value theorem for 
\[ 
\big|\sincs^{(\tilde j)}(\lk k(z-\tilde\zeta\sqrt\gamma-\epsilon(\Delta_0-\xi)))-\sincs^{(\tilde j)}(\lk k(z-\epsilon(\Delta_0-\xi)))\big| \le \lk k|\tilde\zeta|\sqrt\gamma\big\|\sincs^{(\tilde j+1)}\big\|_\infty \]
where $\|\sincs^{(\tilde j+1)}\|_\infty\le C,$ yields
\begin{align*}
&\partial_z^{j}\left(\sincs(\lk k(z-\tilde\zeta\sqrt\gamma-\epsilon(\Delta_0-\xi)))e^{-i\bk k(z-\tilde\zeta\sqrt\gamma-\epsilon(\Delta_0-\xi))}\right) \\
&\qquad= \partial_z^{j}\left(\sincs(\lk k(z-\epsilon(\Delta_0-\xi)))e^{-i\bk k(z-\tilde\zeta\sqrt\gamma-\epsilon(\Delta_0-\xi))}\right) + |\tilde\zeta|\bk k^j\sum_{\tilde j=0}^j\mathcal O\Big((\lk k\bk k^{-1})^{\tilde j}\lk k\sqrt\gamma\Big).
\end{align*}
Again by \autoref{as:three}, using that $\lk k\bk k^{-1}$ is small, we only keep terms of zeroth order and obtain  
\begin{align*}
I_{j,\epsilon}^{(2)}(z) &= -i^{j-2}\frac{\lk k}{\sqrt\pi}\Big(\int_{\R}e^{-\frac14\tilde\zeta^2}\partial_z^{j}\Big(\sincs(\lk k(z-\epsilon(\Delta_0-\xi)))e^{-i\bk k(z-\tilde\zeta\sqrt\gamma-\epsilon(\Delta_0-\xi))}\Big) d\tilde\zeta \\
 &\qquad\qquad\qquad+ \bk k^j\mathcal O(\lk k\sqrt\gamma)\int_\R e^{-\frac14\tilde\zeta^2}|\tilde\zeta| d\tilde\zeta\Big ) \\
&= -2i^{j-2}\lk k \left(e^{-\bk k^2\gamma}\partial_z^{j}\left(\sincs(\lk kz-\epsilon(\Delta_0-\xi))e^{-i\bk k(z-\epsilon(\Delta_0-\xi))}\right)+ \bk k^j\mathcal O(\lk k\sqrt{\gamma}) \right). 
\end{align*}
\end{itemize}
Introducing the functions $u_j$ as in \autoref{eq:sinc_der}, we obtain for
\begin{align*}
\mathcal F_{k}(\mathcal I^*_1[n_1]&\bm\chi_\mathcal S)(z) = \left(\frac{r_1^\dagger}{16\pi^3\rho}\right)\sqrt{\frac{\pi}{2}}\sum_{\epsilon\in\{-1,1\}}\left(\frac1{\psi_0}(I_{2,\epsilon}^{(1)}(z)-I_{2,\epsilon}^{(2)}(z))+\frac{2a\epsilon i}{\psi_0^2}(I_{3,\epsilon}^{(1)}(z)-I_{3,\epsilon}^{(2)}(z))\right) \\
&=-\frac{1}{8\pi^2 \sqrt{2\pi}} \frac{\lk k \bk k^2 r_1^\dagger}{\rho\psi_0}\sum_{\epsilon\in\{-1,1\}}e^{-i\bk k(z-\epsilon\Delta_0)}\Bigg[u_2(z-\epsilon\Delta_0) - \frac{2a\bk k\epsilon}{\psi_0} u_3(z-\epsilon\Delta_0)\\
&- e^{-\bar k^2 \gamma}e^{-i \bar k \epsilon\xi}\left( \big(u_2(z-\epsilon(\Delta_0-\xi)) +\mathcal O(\lk k\sqrt\gamma)\big)- \frac{2a\bk k\epsilon}{\psi_0}\big(u_3(z-\epsilon(\Delta_0-\xi)) +\mathcal O(\lk k\sqrt\gamma)\big) \right)\Bigg].
\end{align*}
\quad
\end{proof}

By \autoref{as:three} we have that $\lk k\Delta_0\gg1$, which holds true for the experimental data. We can then ignore for $z>0$ those sinc terms in \autoref{eq:sinc_app} which are centered on the negative axis. The function $\mathcal F_{k}(\mathcal I^*_1[n_1]\bm\chi_\mathcal S)$ is then approximated by 
\begin{align}
\label{eq:assumption}
&\mathcal F_{k}(\mathcal I^*_1[n_1]\bm\chi_\mathcal S)(z) = -\frac1{8\pi^2\sqrt{2\pi}}\frac{\lk k \bk k^2 r_1^\dag}{\rho\psi_0} e^{-i\bar k(z-\Delta_0)}   \Bigg( u_2(z-\Delta_0) - \frac{2a\bk k}{\psi_0} u_3(z-\Delta_0) \\
&- e^{-\bar k^2 \gamma}e^{-i \bar k \xi}\left( (u_2(z-(\Delta_0-\xi)) +\mathcal O(\lk k\sqrt\gamma))- \frac{2a\bk k}{\psi_0}(u_3(z-(\Delta_0-\xi)) +\mathcal O(\lk k\sqrt\gamma)) \right)+\mathcal O\left(\frac1{\lk k\Delta_0}\right)  \Bigg),\nonumber
\end{align}  
for $z>0$, which reflects the minor influence of the sinc-functions $u_j$ which are centered around $-\Delta_0$ at the function value close to the point $z=\Delta_0.$

\begin{lemma} 
\label{lem:dom_terms} 
\begin{enumerate}
\item
For $z+\Delta_0>0$, we have the asymptotic behavior
\begin{align}
\label{eq:dom_terms}
&\left(\mathcal F_{k}(\mathcal I^*_1[n_1]\bm \chi_\mathcal S)\right)(z + \Delta_0) e^{i\bar k z} \nonumber\\
&\qquad= -\frac1{8\pi^2\sqrt{2\pi}}\frac{\lk k\bk k^2r_1^\dagger}{\rho\psi_0} \Bigg[ -\sincs\left(\lk k z\right)+ e^{-\bar{k}^2\gamma}e^{-i \bar k \xi}\sincs\left(\lk k (z+\xi)\right) \\
&- 2i \Big(  \lk k\bk k^{-1}  \big(\sincs'\left(\lk k z\right) - e^{-\bar{k}^2\gamma}e^{-i \bar k \xi} \sincs'\left(\lk k (z+\xi)\right)\big)+ \frac{a\bk k}{\psi_0}\big(\sincs\left(\lk k z\right) -e^{-\bar{k}^2\gamma}e^{-i \bar k \xi} \sincs\left(\lk k (z+\xi)\right)\big)\Big) \nonumber\\
&\qquad+\mathcal O\left(\left(\bk k^{-1}\lk k+\frac{a\bk k}{\psi_0}\right)^2\right)+ \mathcal O(\lk k\sqrt\gamma) \left( 1+ \frac{a\bk k}{\psi_0}\right)+\mathcal O\left(\frac1{\lk k\Delta_0}\right)\Bigg].\nonumber
\end{align}
\item
In particular, we find for the norm in the highest order
\begin{equation}
\label{eq:dom_terms_squared}
\begin{split}
\left|\left(\mathcal F_{k}(\mathcal I^*_1[n_1]\bm \chi_\mathcal S)\right)(z + \Delta_0)\right|^2 &= \frac1{2^7\pi^5}\left(\frac{\un k\bk k^2r_1^\dagger}{\rho\psi_0}\right)^2 \Big| -\sincs\left(\lk k z\right)+ e^{-\bar{k}^2\gamma}e^{-i \bar k \xi}\sincs\left(\lk k (z+\xi)\right) \\
&\qquad+\mathcal O\left(\frac1{\lk k\Delta_0}\right)+\mathcal O(\bk k^{-1}\lk k)+\mathcal O\left(\frac{a\bk k}{\psi_0}\right)+\mathcal O(\lk k\sqrt{\gamma})\left(1+\frac{a\bk k}{\psi_0}\right)\Big|^2.
\end{split}
\end{equation}
\end{enumerate}
\end{lemma}
\begin{proof}
\begin{enumerate}
\item
We remark that
\begin{align*}
u_j(z) &= \bk k^{-j}\partial_z^j\left(\sincs(\lk kz)e^{-i\bk kz}\right)e^{i\bk kz} \\
&= \bk k^{-j}\sum_{l=0}^j\binom jl\partial_z^l(\sincs(\lk kz))\partial_z^{j-l}(e^{-i\bk kz})e^{i\bk kz} \\
&= (-i)^j\big(\sincs(\lk kz)+ij\bk k^{-1}\lk k\sincs'(\lk kz)+\mathcal O(\bk k^{-2}\lk k^2)\big).
\end{align*}

With this, we find
\begin{align*}
\Big(u_2(&z) - \frac{2a\bk k}{\psi_0}  u_3(z) \Big) -e^{-\bar k^2 \gamma}e^{-i \bar k \xi}\Big( u_2(z+\xi) - \frac{2a\bk k}{\psi_0}u_3(z+\xi) \Big) \\
&=-\sincs(\lk kz)-2i\bk k^{-1}\lk k\sincs'(\lk kz)-\frac{2a}{\psi_0}\big(i\bk k\sincs(\lk kz)-3\lk k\sincs'(\lk kz)\big) \\
&+e^{-\bar k^2 \gamma}e^{-i \bar k \xi}\left(\sincs(\lk k(z+\xi))+2i\bk k^{-1}\lk k\sincs'(\lk k(z+\xi))+\frac{2a}{\psi_0}\big(i\bk k\sincs(\lk k(z+\xi))-3\lk k\sincs'(\lk k(z+\xi))\big)\right)\\
&\qquad\qquad + \mathcal O(\bk k^{-2}\lk k^2).
\end{align*}

If we neglect herein the fourth and eighth term
\[ \frac{6a\lk k}{\psi_0}\big(\sincs'(\lk kz)-e^{-\bar k^2 \gamma}e^{-i \bar k \xi}\sincs'(\lk k(z+\xi))\big) = \mathcal O\left(\frac{a\bk k}{\psi_0}\,\bk k^{-1}\lk k\right) \]
as they are of the order of the small quantity $\frac{a\bk k}{\psi_0}$ (which is of order $10^{-3}$ in our setting) and plug this into our expression in \autoref{eq:assumption} for $\mathcal F_{k}(\mathcal I^*_1[n_1]\bm \chi_\mathcal S)$, we end up with \autoref{eq:dom_terms}.
\item
We see that the second term in \autoref{eq:dom_terms} is of lower order compared to the first one:
\begin{align}
\label{eq:lower_order} 
\lk k\bk k^{-1}  \big(\sincs'\left(\lk k z\right) - e^{-\bar{k}^2\gamma}e^{-i \bar k \xi}\sincs'&(\lk k(z+\xi))\big) \\ 
&+ \frac{a\bk k}{\psi_0}\big(\sincs\left(\lk k z\right) - e^{-\bar{k}^2\gamma}e^{-i \bar k \xi} \sincs(\lk k (z+\xi))\big) = \mathcal O(\bk k^{-1}\lk k)+\mathcal O\left(\frac{a\bk k}{\psi_0}\right). \nonumber
\end{align}
If we thus neglect this term, we obtain \autoref{eq:dom_terms_squared}.
\end{enumerate}
\end{proof} 

In order to apply these formulas for $j \geq 2,$ we notice that the only difference from \autoref{eq:zero_order} is the term $\displaystyle\sum_{l=1}^{j-1} 2 n_{l} d_l\cos\theta^l_t, $
where the components are calculated iteratively by using the linearization of the square-root for $\kappa\in\mathcal B=D_{k\sin\theta}(0)$ making use of the smallness of the angle of acceptance $\theta$:
\begin{align*}
2n_{1}d_1 \cos \theta_t^{1} &= 2n_{1}d_1 \sqrt{1 - \frac{|\kappa|^2 n_0^2}{n_1^2k^2} } = 2n_1 d_1 - \frac{n_0^2}{n_1}\frac{|\kappa|^2 }{k^2} d_1+\mathcal O(\theta^4), \quad l=1, \\
2n_{l}d_l \cos \theta_t^{l} &= 2n_l d_l - \frac{n_0^2}{n_l}\frac{|\kappa|^2 }{k^2} d_l+\mathcal O(\theta^4), \quad l \geq 2.
\end{align*}

Thus, we can adapt \autoref{eq:dom_terms_squared} to the contribution $\mathcal I^*_j$ from the $j$th interface by introducing the layer dependent variables
\begin{equation}
\label{eq:new_variables}
\begin{aligned}
\Delta_j &= \Delta_0 + 2 \sum_{l=1}^{j-1} n_l d_l, \\
\psi_{0,j} &= \psi_0 + 2\sum_{l=1}^{j-1} \frac{d_l n_0}{n_l}, \\
\xi_j &= \frac{\psi_{0,j}}{2}\sin^2(\theta),
\end{aligned}
\end{equation}
for all $j=1,2,\ldots,J$, where the coefficients for the first interface coincide with the system parameters: $\Delta_1=\Delta_0$, $\psi_{0,1}=\psi_0$, and $\xi_1=\xi$.

Moreover, we use \autoref{lem:ref_approx} to approximate the term $r_{\le j-1}(\kappa)$, defined in \autoref{eq:phase}, in the integral of $\mathcal I_j^*$ by
\[ r_{\le j-1}(\kappa) = r_{\le j-1}^\dag+\mathcal O(\theta^2)\quad\text{with}\quad r_{\le j-1}^\dag=\prod_{l=1}^{j-1}(1-(r^\dag_l)^2). \]
With this, we obtain more generally the formula 
\begin{align}
\label{eq:dom_terms_squared_j}
\Big|\Big(\mathcal F_{k}&(\mathcal I^*_j[n_j]\bm \chi_\mathcal S)\Big)(z)\Big|^2 \\
&= \frac1{2^7\pi^5}\left(\frac{\un k\bk k^2r_j^\dagger r_{\le j-1}^\dag}{\rho\psi_{0,j}}\right)^2 \Bigg| U_j(z)+\mathcal O\left(\frac1{\lk k\Delta_j}\right)+\mathcal O(\bk k^{-1}\lk k)+\mathcal O\left(\frac{a\bk k}{\psi_{0,j}}\right)+ \mathcal O(\lk k\sqrt{\gamma})\left(1+\frac{a\bk k}{\psi_{0,j}}\right) \Bigg|^2 \nonumber\\
&\qquad\qquad+\left(\frac{\un k\bk k^2}{\rho\psi_{0,j}}\right)^2\mathcal O(\theta^2)\nonumber
\end{align}
with the function
\begin{equation}
\label{eq:reference_function}
U_j(z) = -\sincs(\lk k (z-\Delta_j)) + e^{-\bar k^2\gamma} e^{-i \bk k\xi_j}\sincs(\lk k (z-(\Delta_j - \xi_j))),
\end{equation} 
for the Fourier transform of the forward operator $\mathcal I_j^*$. The last error term $\mathcal O(\theta^2)$ herein accounts for the approximation of the combined reflection coefficient $r_{\leq {j-1}}$ by its zeroth order.     

In particular, we can use this explicit expression to estimate how much the signal $\mathcal F_k(\mathcal I^*_j\bm\chi_\mathcal S)$ from the $j$th interface influences the values $\mathcal F_k(\mathcal I^*_l\bm\chi_\mathcal S)$ from the $l$th interface in an interval $\mathcal U_l$ around the peak $\Delta_l$ of the second interface. We show the estimate for simplicity only for the interaction between the first and the second interface.

\begin{lemma}
\label{lem:distance}
There exists an interval $\mathcal U_2$ such that
\[ \frac{|\mathcal F_k(\mathcal I^*_1\bm\chi_\mathcal S)(z)|}{|\mathcal F_k(\mathcal I^*_2\bm\chi_\mathcal S)(z)|} \le \frac1{\lk kn_1d_1}\frac4{1-e^{-\bk k^2\gamma}}\frac{|r_1^\dagger|}{|r_2^\dagger r_{\le1}^\dag|}\left(1+\frac{2n_0d_1}{n_1\psi_0}\right) = \mathcal O\left(\frac1{\lk kd_1}\right) \]
for all $z\in\mathcal U_2$.
\end{lemma}

\begin{proof}
The behavior of the two Fourier transforms is mainly described by the functions $U_1$ and $U_2$ from \autoref{eq:reference_function}. We therefore estimate the contribution from $U_1$ at a point $z=y+\Delta_2=y+\Delta_1+2n_1d_1$ around the position $\Delta_2$, where $U_2$ has its peak, from above and the contribution of $U_2$ there from below.

To this end, we first choose a small parameter $\tau\in(0,n_1d_1)$ such that
\[ |U_2(z)| = \left|\sincs(\lk k y) - e^{-\bar k^2\gamma} e^{-i \bk k \xi_2}\sincs(\lk k (y +\xi_2))\right| \ge \frac{1-e^{-\bar k^2\gamma}}2 \]
holds for all $y\in(-\tau,\tau)$. (Such a parameter has to exist, since at $y=0$ the term is bounded by $1-\e^{-\bk k^2\gamma}$.)

Since $y\in(-\tau,\tau)$, we can then also bound $U_1$ with
\[ |U_1(z)|=\left|\sincs(\lk k (y+2n_1 d_1)) - e^{-\bar k^2\gamma} e^{-i \bk k \xi_1}\sincs(\lk k (y+2 n_1 d_1 +\xi_1)))\right| \le \frac{1+e^{-\bar k^2\gamma}}{\lk k(2n_1 d_1-\tau)} \le \frac2{\lk kn_1 d_1} \]
from above.

Therefore, we get with $\psi_{0,1}=\psi_0$ and $\psi_{0,2}=\psi_0+\frac{2n_0d_1}{n_1}$
\[ \frac{|\mathcal F_k(\mathcal I^*_1\bm\chi_\mathcal S)(z)|}{|\mathcal F_k(\mathcal I^*_2\bm\chi_\mathcal S)(z)|} = \left|\frac{r_1^\dagger}{r_2^\dagger r_{\le 1}^\dag}\frac{\psi_{0,2}}{\psi_{0,1}}\frac{U_1(z)}{U_2(z)}\right| \le \frac1{\lk kn_1d_1}\frac4{1-e^{-\bk k^2\gamma}}\frac{|r_1^\dagger|}{|r_2^\dagger r_{\le 1}^\dag|}\left(1+\frac{2n_0d_1}{n_1\psi_0}\right). \]
\quad
\end{proof}

The above arguments justify the method to reduce the all-at-once-approach of the inverse problem to the simpler layer-by-layer reconstruction presented in \autoref{sec:layer}. However, the enclosed form of the integral operators $\mathcal I_j$ (and $\mathcal I_j^*$ respectively) still prevents us from an analytic expression for the reconstruction. Hence, we formulate our problem as a least squares minimization problem.

\section{Least Squares Minimization for the Refractive Index}
\label{sec:def_minimization}

To numerically obtain a solution of the inverse problem in \autoref{eq:inv_layer} for a fixed $j$, we write it as a discrete least squares minimization problem of the functional 
\begin{equation}
\label{eq:min_func}
\mathcal J_j(n_j,d_{j-1}) = \sum_{m=1}^{M_j}\left(\left|\mathcal F_k\left(\left(\sum_{l=1}^{j-1} \mathcal I_l+ \mathcal I_j[n_j,d_{j-1}]\right)\bm\chi_\mathcal S\right)(z_{j,m})\right|^2
- y_{j,m}^{(\delta)} \right)^2, 
\end{equation} 
for the parameters $n_j$ and $d_{j-1}$.

To simplify the analysis, we will, however, replace herein the full forward operator $\mathcal I_j$ by the reduced forward operator $\mathcal I_j^*$ and consider thus the functional
\begin{equation}
\label{eq:approx_functional}
\mathcal J^*_j(n_j,d_{j-1}) = \sum_{m=1}^{M_j}\left(\left|\mathcal F_k\left(\left(\sum_{l=1}^{j-1} \mathcal I_l + \mathcal I^*_j[n_j,d_{j-1}]\right)\bm\chi_\mathcal S\right)(z_{j,m})\right|^2
- y_{j,m}^{(\delta)} \right)^2,
\end{equation} 

Hence, for every interface $j$ we want to minimize $\mathcal J^*_j(n_j,d_{j-1})$ with respect to $n_j$ and $d_{j-1}$. Hereby, the data is provided on a discretized grid of $M_j$ points which we denote by $\{z_{j,m}\}_{m=1}^{M_j}\subset\mathcal U_j$ in the interval $\mathcal U_j$ selected in \autoref{sec:layer}. We silently assume that the distance between the interfaces is sufficiently large so that the influence from the values $\mathcal F_k(\mathcal I_l[n_l,d_{l-1}]\bm\chi_{\mathcal S})$ for $l>j$ in the interval $\mathcal U_j$ is negligible (as estimated in \autoref{lem:distance}).

Finally, we declare the set of admissible values for $n$ and $d$ (in every step) by 
\begin{equation}
\label{eq:adm_v}
\mathcal A_j = \{(n,d)\in\R^2\ |\ 1\leq n<\infty,\, n\neq n_{j-1},\ 0< d <\infty\}\subset \R^2.
\end{equation} 
The reconstruction is based on the isolation of the clearly visible peaks in the data, which originates from a relatively high refractive index contrast between the single layers. If no contrast is present, meaning that $n_j=n_{j-1},$ the method considers the layer as a unit and goes on to the next layer. Hence, we may exclude -- for the step $j$ -- $n_{j-1},$ the refractive index from the previously reconstructed layer, from the admissible space $\mathcal A_j.$ 
        
Calculating both values from this minimization functional is a slight overkill, since the width $d_{j-1}$ in the $j$th step, can be found by the largest overlap of the sinc-function in the forward model and the peak in the data independent of their height, which itself is determined by the refractive indices of the different media. Thus, we can separate the extraction of the pair of parameters $(n,d)$ into two parts, similar to how it is done for stepwise-gradient methods. Firstly, the width $d$ of the layer is reconstructed, and secondly the refractive index.
We delay the justification of our arguments to \autoref{sec:width}. There we show (visually) that the functionals (in a step $j$) actually allow a precise prediction of the width $d_{j-1}$ without the knowledge of the refractive index $n_j.$        

We will therefore in the following consider the minimization problems with respect to $n_j$ for a known value $d_{j-1}$. We thus end up with the reduced minimization problem 
\begin{equation*}
\label{eq:min_problem}
\argmin_{n\in\mathcal A_j} \mathcal J^{(*)}_j(n), \quad 1\leq j \leq J+1,  \tag{\textasteriskcentered} 
\end{equation*}
for the functionals defined in \autoref{eq:min_func} and \autoref{eq:approx_functional}. We hereby assume that the value of $d_{j-1}$ has been reconstructed already and restrict the analysis to the refractive index argument. The second argument of the functional is therefore omitted.

To simplify the notation, we define for $j$ and $m$
\begin{equation}
\label{eq:short_ff}
\Lambda_{j,m} = \mathcal F_{k}\left(\sum_{l=1}^{j-1} \mathcal I_l\bm\chi_\mathcal S\right)(z_{j,m}),\quad \Gamma_{j,m}(n_j) = \mathcal F_k(\mathcal I_j[n_j]\bm\chi_\mathcal S)(z_{j,m}),\quad \Gamma^*_{j,m} = \mathcal F_{k}\left(\frac{\mathcal I^*_j[n]}{r^\dagger_j}\bm\chi_\mathcal S\right)(z_{j,m}),
\end{equation}
so that we can write
\[ \left|\mathcal F_k\left(\left(\sum_{l=1}^{j-1} \mathcal I_l + \mathcal I_j[n_j]\right)\bm\chi_\mathcal S\right)(z_{j,m})\right|^2 = \left| \Lambda_{j,m} + \Gamma_{j,m}(n_j)\right|^2 = F_{j,m}(n_j) \]
and
\[ \left|\mathcal F_k\left(\left(\sum_{l=1}^{j-1} \mathcal I_l + \mathcal I^*_j[n_j]\right)\bm\chi_\mathcal S\right)(z_{j,m})\right|^2 = \left| \Lambda_{j,m} + r_j^\dagger(n_j) \Gamma^*_{j,m}\right|^2 = F^*_{j,m}(n_j). \]
We note that by definition $\Gamma^*_{j,m}$ is independent of the refractive index, since $\mathcal I^*_j$ carries the directional independent reflection coefficient $r_j^\dagger,$ see \autoref{eq:integral_approx}.

Before we consider the minimization problem, we want to justify that the modeling error introduced by replacing $\mathcal J_j$ with the simplified functional $\mathcal J_j^*$ can be controlled.

\begin{lemma}
\label{lem:func_estimate}
For $2\leq j\leq J+1$ let $\mathcal J_j$ and $\mathcal J^*_j$ be defined by \autoref{eq:min_func} and \autoref{eq:approx_functional}, respectively. Then for a fixed value $d,$ the following estimate holds true for any $n$
\[ 
|\mathcal J_j(n) - \mathcal J^*_j(n)| \leq C \|\theta_t^{j-1}\|^2_{\infty,\tilde{\mathcal B}}, 
\]
where $\theta_t^{j-1}$ is the (small) angle of transmission (defined in \autoref{eq:angleoft}) and where we defined the supremum norm $\|h\|_{\infty,\tilde{\mathcal B}} = \sup_{\tilde\kappa\in \tilde{\mathcal B}}|h(\tilde\kappa)|$ on the reduced domain of integration $\tilde{\mathcal B},$ which is defined by $\tilde \kappa = \frac{\kappa}{k}$ for $\kappa \in\mathcal B$.
\end{lemma}

\begin{proof}
We first calculate the difference, which using $F_{j,m} = F_{j,m}-F^*_{j,m}+F^*_{j,m}$ is given by 
\begin{multline*}
\mathcal J_j(n) - \mathcal J^*_j(n) = \sum_{m=1}^{M_j} (F_{j,m}(n) - F^*_{j,m}(n))\left(F_{j,m}(n) + F^*_{j,m}(n) - 2 y_{j,m}^{(\delta)}\right) \\
= \sum_{m=1}^{M_j} (F_{j,m}(n) - F^*_{j,m}(n))^2 + 2\left(F^*_{j,m}(n) - y_{j,m}^{(\delta)}\right)(F_{j,m}(n) - F^*_{j,m}(n)) .
\end{multline*}
By using the definition of the forward model, we find that the essential part is to find an upper bound for
\begin{multline*}
F_{j,m} - F^*_{j,m} = 2 \Re\left\{\Lambda_{j,m}\overline{(\Gamma_{j,m}-r_j^\dagger\Gamma^*_{j,m})}\right\} + \left(\left|\Gamma_{j,m}\right|^2 - \left|r_j^\dagger\Gamma^*_{j,m}\right|^2 \right)\\
\leq 2 |\Lambda_{j,m}|\left|\Gamma_{j,m}-r_j^\dagger\Gamma^*_{j,m}\right| + \left(\left|\Gamma_{j,m}\right|^2 - \left|r_j^\dagger\Gamma^*_{j,m}\right|^2 \right).
\end{multline*}
First, we need to estimate the integral 
\[
\frac{1}{16 \pi^3 \rho }\int_{\mathcal S}k\int_{\mathcal B} e^{-|\kappa|^2 a}d \kappa d k.
\] 
After a change of coordinates $\kappa = k \tilde \kappa,$ where $\tilde\kappa$ is in the reduced set $\tilde{\mathcal B}$, we can replace $\tilde{\mathcal B}$ by the disk $D_{\sin(2\theta_\Omega+\theta)}$ in the integral and find an upper bound
\[
\frac{1}{16 \pi^3 \rho }\int_{\mathcal S}k\int_{\tilde{\mathcal B}} k^2 e^{-k^2|\tilde\kappa|^2 a}d \tilde\kappa d k \leq \frac{ 1 }{32 \pi^2  \sqrt{2\pi}\rho a}\left( k^2\Big|_{\partial \mathcal S} + \frac{e^{-\sin^2(2\theta_\Omega+\theta)k^2 a}\Big\vert_{\partial \mathcal S}}{\sin^2(2\theta_\Omega+\theta) a} \right) = U_p.
\]
For any value of $n$ it holds that the supremum norm $\|r_{\leq j-1}\|_{\infty,\tilde{\mathcal B}} = \sup_{\tilde\kappa\in \mathcal B} |r_{\leq j-1}(\tilde\kappa)|\leq 1.$ Thus, we obtain
\[
|\Gamma_{j,m}(n)| \leq  \sup_n \|r_j(.\, ,n)\|_{\infty,\tilde{\mathcal B}} U_p,\quad |r_j^\dagger(n)\Gamma^*_{j,m}| \leq \sup_n |r^\dagger_j(n)| U_p .    
\]
The difference between the actual function and its approximated form then is estimated by 
\begin{align*}
\int_{\mathcal S} \left|\mathcal I_j[n]- \mathcal I^*_j[n]\right| d k &\leq \frac{1}{16\pi^3\rho}\int_{\mathcal S} k \int_{\mathcal B} \left|r_j(\kappa,n)-r_j^\dagger(n)\right| e^{-|\kappa|^2 a} |r_{\leq j-1}(\kappa)| d\kappa \\
& \leq \frac{\|\theta_t^{j-1}\|^2_{\infty,\tilde{\mathcal B}}}{2} U_p.
\end{align*}
We combine all the above estimates to find
\begin{align*}
\left|\Gamma_{j,m}\right|^2 - \left|r_j^\dagger\Gamma^*_{j,m}\right|^2 &\leq \left|\Gamma_{j,m}-r_j^\dagger \Gamma^*_{j,m}\right| \left(|\Gamma_{j,m}|+\left|r_j^\dagger \Gamma^*_{j,m}\right|\right)\\
&\leq \left(\frac{\|\theta_t^{j-1}\|^2_{\infty,\tilde{\mathcal B}}}{2}\left(\sup_n \|r_j(.\, ,n)\|_{\infty,\tilde{\mathcal B}} +\sup_n |r_j^\dagger(n)|\right)\right)U_p^2.
\end{align*}
Finally, we find an upper bound for
\begin{align*}
2|\Lambda_{j,m}| \left|\Gamma_{j,m}-r_j^\dagger\Gamma^*_{j,m}\right| &= 2 \left| \sum_{l=1}^{j-1}\mathcal F_k\left(\mathcal I_l\bm\chi_\mathcal S\right)(z_{l,m})\right| \left|\Gamma_{j,m}-r_j^\dagger\Gamma^*_{j,m}\right| \\
&\leq \left(\sum_{l=1}^{j-1} \|r_l\|_{\infty,\tilde{\mathcal B}}\|\theta_t^{j-1}\|^2_{\infty,\tilde{\mathcal B}}\right) U_p^2.
\end{align*}
Thus, we get 
\begin{multline*}
|\mathcal J_j(n) - \mathcal J^*_j(n)| \leq \sum_{m=1}^{M_j} (F_{j,m}(n) - F^*_{j,m}(n))^2 + 2 \sum_{m=1}^{M_j} \left|\left(F^*_{j}(n) - y_j^{(\delta)}\right)_m\right| \left|F_{j,m}(n) - F^*_{j,m}(n)\right|\\
\leq M_j \left(\left(\sum_{l=1}^{j-1} \|r_l\|_{\infty,\tilde{\mathcal B}}\|\theta_t^{j-1}\|^2_{\infty,\tilde{\mathcal B}}\right) + \frac{\|\theta_t^{j-1}\|^2_{\infty,\tilde{\mathcal B}}}{2}\left(\sup_n \|r_j(.\, ,n)\|_{\infty,\tilde{\mathcal B}} +\sup_n |r_j^\dagger(n)|\right)\right)^2 U_p^4 \\+ 2 U_p^2 M_j \sup_m \left(\sup_n |(F^*_j(n) - y^{(\delta)}_j)_m|\right)  \Bigg(\left(\sum_{l=1}^{j-1} \|r_l\|_{\infty,\tilde{\mathcal B}}\|\theta_t^{j-1}\|^2_{\infty,\tilde{\mathcal B}}\right)\\
 +\frac{\|\theta_t^{j-1}\|^2_{\infty,\tilde{\mathcal B}}}{2}\left(\sup_n \|r_j(.\, ,n)\|_{\infty,\tilde{\mathcal B}} +\sup_n |r_j^\dagger(n)|\right)\Bigg). 
\end{multline*}
We summarize all coefficients of $\|\theta_t^{j-1}\|^2_{\infty,\tilde{\mathcal B}}$ in a constant $C$ and obtain the desired result.
\end{proof}
The estimate and its proof for $j=1$ is similar 
\begin{multline*}
|\mathcal J_1(n) - \mathcal J^*_1(n)| \leq M_1 \left( \frac{\|\theta_t^{0}\|^2_{\infty,\tilde{\mathcal B}}}{2}\left(\sup_n \|r_1(.\, ,n)\|_{\infty,\tilde{\mathcal B}} +\sup_n |r_1^\dagger(n)|\right)\right)^2 U_p^4 \\+ 2 U_p^2 M_1 \sup_m \left(\sup_n |(F^*_1(n) - y^{(\delta)}_1)_m|\right)  \Bigg( \frac{\|\theta_t^{0}\|^2_{\infty,\tilde{\mathcal B}}}{2}\left(\sup_n \|r_1(.\, ,n)\|_{\infty,\tilde{\mathcal B}} +\sup_n |r_1^\dagger(n)|\right)\Bigg).
\end{multline*}

\section{Existence and Uniqueness Results}
\label{sec:inverse_EaU}
We finally turn to the question if the minimization problem \eqref{eq:min_problem} has a unique minimizer. Actually, under the assumption that the data is an element of the range of the forward model, that is for any $m$ we have
\begin{equation}
\label{eq:data_range}
y_{j,m} = |\Lambda_{j,m} + r^\dagger_j(\tilde n) \Gamma^*_{j,m}|^2,\quad \text{for some}\quad \tilde n\in\mathcal A_j, 
\end{equation}  
we can show that there exists a unique solution to the minimization problem \eqref{eq:min_problem} corresponding to the approximated functional in \autoref{eq:approx_functional}. 

For $j=1,$ we have that $\Lambda_{1,m} = 0,$ for all $m,$ which reduces the functional $\mathcal J^*_1$ to 
\[
\mathcal J^*_1(n) = \sum_{m=1}^{M_1}\left(\left|r^\dagger_1(n)\Gamma^*_{1,m} \right|^2 -y^{(\delta)}_{1,m}\right)^2. 
\] 
This functional, even in the best case, where the data is in the range of the forward model, attains two global maxima. In order to exclude one of the possible solutions, we use the fact that the object is surrounded by air and that the refractive index of the first layer thus has to satisfy $n_1 > 1.$ 

For $j\geq 2,$ we obtain the uniqueness from the definition of the functional, which allows -- in contrast to $j=1$ -- the calculation of the reflection coefficient from a second order polynomial equation.  
For this reason, we discuss the cases $j=1$ and $j\geq 2$ separately.

\begin{theorem}
\label{thm:unique_min}
For $j=1$ let $\mathcal J^*_1$ be defined as in \autoref{eq:approx_functional} and let $y_1$ satisfy the range condition \autoref{eq:data_range} for some $\tilde n_1 >1$. Then there exists a unique solution to the minimization problem \eqref{eq:min_problem}.
\end{theorem}

\begin{proof}
We take the derivative with respect to $n$ and find
\[
\partial_n \mathcal J^*_1(n) = 4 r_1^\dagger(n) \partial_n r_1^\dagger(n) \sum_{m=1}^{M_1} \left(\left|\Gamma^*_{1,m}\right|^2  \left(\big|r_1^\dagger(n)\big|^2 - \big|r_1^\dagger(\tilde n_1)\big|^2\right)\right) \big|\Gamma^*_{1,m}\big|^2.
\]
Since we have $\partial_n r_1^\dagger(n) = \frac{-2}{(1+n)^2} \neq 0$ the derivative of $\mathcal J^*_1$ is zero if and only if $r_1^\dagger(n) = 0$ which implies $n = 1$ or if $r_1^\dagger$ satisfies
\begin{equation}
\label{eq:refl_square}
 \big |r_1^\dagger(n)\big|^2 =  \big |r_1^\dagger(\tilde n_1)\big|^2.
\end{equation}
Since we have excluded $n = 1$ (as refractive index from the previous layer) from the possible solutions, we find that $r_1^\dagger(n) = \pm r_1^\dagger(\tilde n_1),$ which further yields 
\[
n_{\pm} = \frac{1\mp r_1^\dagger(\tilde n_1)}{1\pm\tilde r_1^\dagger(\tilde n_1)}.
\]
Since $\sign(r_1^\dagger(\tilde n_1)) = -1,$ due to the assumption that $\tilde n_1 >1,$ we can also exclude $n_-$ from the set of possible solutions. Otherwise it would hold that $n<1.$ 

The non-negativity of the functional $\mathcal J^*_1$ and the fact that $\mathcal J^*_1(n_+) = 0$ implies that $n_+$ is actually the global minimum.  
\end{proof}

\begin{lemma}
Let $n$ be a minimum of $\mathcal J^*_1,$ then we have $\partial^2_n\mathcal J^*_1(n) = \mathcal P >0.$
\end{lemma}

\begin{proof}
Let $n$ satisfy $\partial_n \mathcal J^*_1(n) = 0.$ Then according to the proof of \autoref{thm:unique_min} the second derivative of $\mathcal J^*_1$ with respect to $n$ is given by 
\[
\partial^2_n\mathcal J^*_1(n) = 8\sum_{m=1}^{M_1} \left|\Gamma^*_{1,m}\right|^4 (r_1^\dagger \partial_n r_1^\dagger )^2 = \mathcal P.
\]
Thus, together with $r_1^\dagger \partial_n r_1^\dagger  = \frac{1-n}{(1+n)^3} \neq 0,$ we immediately find that 
\[
\partial^2_n\mathcal J^*_1(n) = \mathcal P >0.
\]
\end{proof}

We formulate the uniqueness for the deeper interfaces in the case $j=2$. The generalization to general $j$ is straightforward.

\begin{proposition}
\label{lem:loc_extrema}
Let the functional $\mathcal J^*_2$ be defined as in \autoref{eq:approx_functional} and let the data $y_2$ satisfy the range condition \autoref{eq:data_range} for some $\tilde r_2:=r_2^\dagger(\tilde n_2)\in[-1,\,1].$ Then, for a suitable choice of discretization points $z_{2,m}$, $m\in\{1,\ldots,M_2\}$ with $M_2>1$, the value $\tilde n_2$ is the unique minimizer of $\mathcal J^*_2$:
\[ \mathcal J^*_2(\tilde n_2) = 0\quad\text{and}\quad\mathcal J^*_2(n) > 0\quad\text{for all}\quad n\ne\tilde n_2. \]
\end{proposition}

\begin{proof}
Using the abbreviations from \autoref{eq:short_ff}, the non-negative functional $\mathcal J^*_2$ can be written as
\begin{equation}
\label{eq:redefine}
\begin{split}
\mathcal J^*_2(n) &= \sum_{m=1}^{M_2}\left(\left|\Gamma^*_{2,m}\right|^2((r_2^\dag(n))^2 - \tilde r_2^2) + 2\Re\{\Lambda_{2,m}\overline{\Gamma^*_{2,m}}\}(r_2^\dag(n) - \tilde r_2) \right)^2 \\
&= (r_2^\dag(n) - \tilde r_2)^2\sum_{m=1}^{M_2}\left(\left|\Gamma^*_{2,m}\right|^2(r_2^\dag(n) + \tilde r_2) + 2\Re\{\Lambda_{2,m}\overline{\Gamma^*_{2,m}}\} \right)^2 
\end{split}
\end{equation}
Thus, $\mathcal J^*_2$ vanishes if either $r_2^\dag(n)=\tilde r_2$ or
\[ \sum_{m=1}^{M_2}\left(\left|\Gamma^*_{2,m}\right|^2(r_2^\dag(n) + \tilde r_2) + 2\Re\{\Lambda_{2,m}\overline{\Gamma^*_{2,m}}\} \right)^2 = 0. \]
Since all terms in this sum are non-negative, it can only vanish if all of them are zero, that is, if we have
\begin{equation}\label{eq:two_points}
r_2^\dag(n) + \tilde r_2 = -\frac{2\Re\{\Lambda_{2,m}\overline{\Gamma^*_{2,m}}\}}{\left|\Gamma^*_{2,m}\right|^2} = -2\Re\frac{\Lambda_{2,m}}{\Gamma^*_{2,m}}
\end{equation}
for all $m\in\{1,\ldots,M_2\}$ where $\Gamma^*_{2,m}\ne0$.

Inserting the definitions of $\Lambda_{2,m}$ and $\Gamma^*_{2,m}$, this becomes the condition that the function
\[ z\mapsto\Re\frac{\left(\mathcal F_{k}(\mathcal I_1\bm \chi_\mathcal S)\right)(z)}{\left(\mathcal F_{k}\left(\frac{\mathcal I_2^*[n]}{r_j^\dag(n)}\bm \chi_\mathcal S\right)\right)(z)} \]
is constant, at least restricted to the values $z_{2,m}$. However, we see from the asymptotic expression \autoref{eq:dom_terms} for $\mathcal F_{k}(\mathcal I^*_1[n_1]\bm \chi_\mathcal S)$ and the correspondingly adapted expression for the second interface (as in \autoref{eq:dom_terms_squared_j}) that this function (in the leading order terms) is not constant. Therefore, for a suitable choice of discretization points $z_{2,m}$, we can guarantee that there are at least two values $m_1$, $m_2$ for which
\[ 2\Re\frac{\Lambda_{2,m_1}}{\Gamma^*_{2,m_1}}\ne2\Re\frac{\Lambda_{2,m_2}}{\Gamma^*_{2,m_2}}, \]
thus excluding any zero values of $\mathcal J_2^*$ other than $\tilde r_2$.
\end{proof}

\begin{remark}
So far we have considered the case where the data is in the range of the forward model without any noise. Noisy data is treated similarly. We note that since the positive functional $\mathcal J^*_j$ (also with noise) remains a fourth order polynomial in $r^\dagger_j,$ the existence of the local minima with respect to $n$ is guaranteed. One is located on the half $(1,n_{j-1}),$ the other one on $(n_{j-1},\infty).$ The functional $\mathcal J_j^*$ (as a continuous function of $n$) changes also continuously depending on how the data is disturbed. This means, that up to a certain (small) level of noise, the position of the global minimum remains in the correct position. However, cases where the noise is such that the position of the global minimum might jump to the other side cannot be excluded.    
\end{remark}

\subsection*{Sign determination} 
A difficulty, which we are facing in the reconstruction of the first layer in \autoref{thm:unique_min} (and which has been identified also in the reconstruction method shown in \cite{ElbMinVes21}) is that the minimum of the functional can only be determined up to its absolute value. In \cite{ElbMinVes21}, we had in each step two local (and also global) minima, where we excluded one by using a priori information on the refractive indices of the layers.
  
Additional noise on the data may lead to the case where the functional (for the refractive reconstruction from layers deeper inside the object) attains equal values for the two local minima. In this case, a distinction of the correct solution to the minimization problem \eqref{eq:min_problem} is no longer possible. Now we want to show from a very theoretical point of view how one may exclude one of the possible solutions.
     
In detail, let $j\geq 2$ and let the refractive indices and the widths $n_l,d_{l-1},\,1\leq l\leq j-1,$ be determined and assume that the data in the $j$th step is within the range of the forward operator $\mathcal I^*_j$ for some 
\[
\tilde r_j = r_j^\dagger(\tilde n_j) = \frac{n_{j-1}-\tilde n_j}{n_{j-1}+\tilde n_j}.
\] 
Hence, $\tilde r_j$ (or $\tilde n_j$) is the objective to be reconstructed in this step. What we have seen so far is that the minima with respect to $n$ are distributed such that one is located on each side of the point $n_{j-1},$ the refractive index from the previous layer. 

Then, the position of $\tilde n_j$ with respect to $n_{j-1},$ in detail the difference $n_{j-1}-\tilde n_j,$ determines 
\[
\begin{cases} \sign(\tilde r_j)=1,&\text{if}\quad \tilde n_j < n_{j-1},\\
 \sign(\tilde r_j)=-1,&\text{if}\quad \tilde n_j > n_{j-1}.
\end{cases} 
\]
Thus, the correct solution $\tilde n_j$ of the minimization problem -- in step $j$ -- is completely connected to the sign of the reflection coefficient.

If one is able to determine the sign of the reflection coefficient $\tilde r_j$ from the data, one already the determines the position of the exact solution with respect to $n_{j-1}.$ Therefore, while searching for the minimum of the functional, we may restrict to a certain half, $(1,n_{j-1})$ or $(n_{j-1},\infty),$ of the admissible values.

Basically the goal is to read off the sign from the measurement data by avoiding the highly frequent parts originating from the exponential factor $e^{-i \bk k\,.\,}$, see \autoref{eq:inv_layer}. To do so, we multiply the equation with the exponential factor with opposite sign and obtain the sign of $\tilde r_j$ in the neighbourhood of $z = \Delta_j.$
We present this method for small values $\theta_\Omega,$ where we by using \autoref{lem:dom_terms} can calculate the Fourier transform of the integral in \autoref{eq:integral_approx} explicitly in the leading order.
For simplicity, we consider in the following the single reflection from the first interface, meaning we consider the case $j=1.$ 
\begin{lemma}
\label{lem:sign_determination}
Let the asymptotic behavior as in \autoref{lem:dom_terms} hold true, meaning that the Fourier transform of the data, $\mathcal F_k\left(\mathcal C^1\bm\chi_\mathcal S\right),$ is given by \autoref{eq:dom_terms} for a $r_1^\dagger.$ Further, let $\Delta_1$ and $\xi_1$ be defined as in \autoref{eq:new_variables}. Then we obtain 
\begin{align}
\label{eq:real_part}
\Re\Big\{\mathcal F_k\left(\mathcal C^1\bm\chi_\mathcal S\right)&(z+\Delta_1)e^{i \bar k z}\Big\} = \frac{1}{8\pi^2\sqrt{2\pi}}\frac{\lk k \bk k^2 r_1^\dagger}{\rho \psi_0}\Bigg(\sincs(\lk kz) - e^{-\bk k^2\gamma}\cos(\bk k\xi_1)\sincs(\lk k(z+\xi_1)) \\
&+ \mathcal O(\lk k\sqrt\gamma)\left(1+\frac{a\bk k}{\psi_0}\right) + \mathcal O\left(\frac{1}{\lk k\Delta_1}\right) +\mathcal O\left(\frac{a\bk k}{\psi_0}\bk k^{-1} \lk k\right) + \mathcal O(\bk k^{-1}\lk k)\Bigg)  \nonumber
\end{align}
for $z+\Delta_1>0.$ Furthermore, in a sufficiently small neighbourhood of $z=0$ we have that 
\[
\sign(r_1^\dagger) =  \sign\left(\Re\left\{\mathcal F_k\left(\mathcal C^1 \bm\chi_\mathcal S\right)(z+\Delta_1) e^{i\bar k z}  \right\}\right).
\]
\end{lemma}

\begin{proof}
Using the real part of \autoref{eq:dom_terms}, we see that 
\[
\bk k^{-1}\lk k \big(\sincs'(\lk k(z+\xi_1)) + \sincs(\lk k(z+\xi_1))\big)\sin(\bk k \xi_1)e^{-\bk k^2\gamma} = \mathcal O(\bk k^{-1}\lk k)
\] 
is of lower order compared to 
\[
-\sincs(\lk kz) + e^{-\bk k^2\gamma}\cos(\bk k\xi_1)\sincs(\lk k(z+\xi_1)).
\]
Hence, we have that 
\begin{align*}
\Re\Big\{\mathcal F_k\left(\mathcal C^1\bm\chi_\mathcal S\right)&(z+\Delta_1)e^{i \bar k z}\Big\} = \frac{1}{8\pi^2\sqrt{2\pi}}\frac{\lk k \bk k^2 r_1^\dagger}{\rho \psi_0}\Bigg(\sincs(\lk kz) - e^{-\bk k^2\gamma}\cos(\bk k\xi_1)\sincs(\lk k(z+\xi_1)) \\
&+ \mathcal O(\lk k\sqrt\gamma)\left(1+\frac{a\bk k}{\psi_0}\right) + \mathcal O\left(\frac{1}{\lk k\Delta_1}\right) +\mathcal O\left(\frac{a\bk k}{\psi_0}\bk k^{-1} \lk k\right) + \mathcal O(\bk k^{-1}\lk k)\Bigg) 
\end{align*}
Then, in a sufficiently small neighbourhood of $z=0,$ the sign of the dominant term is determined by the sign of the first sinc-function, meaning that
\begin{equation*}
 \sign\left(\sincs\left(\lk k z\right) - e^{-\bar{k}^2\gamma} \cos(\bk k \xi_1)\sincs(\lk k(z +\xi_1))\right) = \sign\big(\sincs\left(\lk k z\right)\big) = 1.
\end{equation*}
Finally, due to the assumption that $\rho,\, \psi_0,\, \lk k,\, \bk k >0,$ together with \autoref{eq:real_part}, we can conclude that within a sufficiently small domain around $z = 0$ 
\[
\sign(r_1^\dagger) =  \sign\left(\Re\left\{\mathcal F_k\left(\mathcal C^1 \bm\chi_\mathcal S\right)(z+\Delta_1) e^{i\bar k z}  \right\}\right).
\] 
\quad
\end{proof}
In the best case possible, where the measurement data (on basis of a plane wave model) is known for all wavenumbers $k\in\R,$ the position $\Delta_1$ is perfectly determined by the duality of the Fourier transform
\[
\mathcal C^1(k) = r_1^\dagger \cos(k \Delta_1)\quad \longleftrightarrow\quad \mathcal F_{k}(\mathcal C^1)(z)\chi_{\R_+}(z) = \frac{r_1^\dagger}{2\sqrt{2\pi}} \delta(z-\Delta_1). 
\] 
However, for the Gaussian model, see \autoref{eq:dom_terms}, the situation is different. In order to determine $\Delta_1,$ the position of the peak, from the measurement, the absolute value of \autoref{eq:dom_terms} is used. There, the sinc-term centered around $z = \Delta_1 - \xi_1$ overlaps (and interacts) with the one centered around $z=\Delta_1$ and therefore eventually shifts the position of the global maximum slightly to $\Delta_1 + \mu_0$ for a $\mu_0\ll 1.$ Even for a small change, say $\bar k \mu_0 > \pi/2,$ we then get from \autoref{eq:real_part} 
\[
\sign\left(\Re\left\{\mathcal F_k\left(\mathcal C^1 \bm\chi_\mathcal S\right)(z) e^{i\bar k (z-(\Delta_1+\mu_0))}  \right\} \right)= \sign\left(\Re\left\{ \sincs(\lk k (z-\Delta_1))e^{-i\bk k\mu_0} \right\}\right) = -\sign\left( \sincs(\lk k (z-\Delta_1)) \right). 
\]
An increasing bandwidth $2\lk k$ of wavenumbers, which yields that 
\[
\lk k\sincs (\lk k z) \longrightarrow \delta(z), \quad \lk k \to \infty,
\]
increases the precision which is needed for determining $\Delta_1.$ However, for an actual OCT setup, the precision is limited and provides only non-sufficient reconstructions.    

\section{Behavior of the Functional with Respect to the Width}
\label{sec:width}

In the previous sections, we showed that in every step the minimization problem \eqref{eq:min_problem} for the functional $\mathcal J^*_j$ in \autoref{eq:approx_functional} attains a unique solution. Hereby, we have assumed that the determination of the width $d_{j-1}$ (in step $j$) can be carried out independently of $n$ and therefore can be assumed to be reconstructed before the $j$th step, $j>1$. 

In this section we want to justify our argument, that this reconstruction may be separated, by presenting the behavior of the corresponding functionals for different values with respect to $d.$ That is, we want to argue that the functional $\mathcal J_j^{(*)}(n,d)$ in the $j$th step (almost) independently of the refractive index $n$ yields a unique minimum with respect to $d.$ 

We have seen that the actual data is nicely predicted by the forward model showing also the sinc-function structure. To this end let us assume that the data $y$ be given and in the range of the forward model, meaning that for $j$ and $m$ let 
\[
y_{j,m} = \left|\Lambda_{j,m} + \mathcal F_k\left(\mathcal I^{*}_{j}[\,\tilde n_j,\tilde d_{j-1}\,]\bm\chi_\mathcal S\right)(z_{j,m})\right|^2,
\] 
for some $(\tilde n_j,\tilde d_{j-1})\in\mathcal A_j$ and let $\tilde r_j = r_j^\dagger(\tilde n_j)$ denote the corresponding directional independent reflection coefficient. Clearly, we obtain the best match, meaning that the according minimization functional is zero, if the amplitude, which is determined by the refractive index $n$, and the width $d$ satisfy $(n,d) = (\tilde n_{j},\tilde d_{j-1}).$ 

In the actual situation the refractive index in the $j$th step is reconstructed after the width of the layer is determined. This means that the reflection coefficient corresponding to the forward model $\mathcal I_j^{*}$ in \autoref{eq:approx_functional}, which further determines the absolute value after Fourier transform, is certainly not matched. We want to show that even though the amplitude is not matched perfectly, the functional (and the corresponding minimization problem with respect to $d$) still yields a unique minimum at the correct position. 

We cannot explicitly calculate the extremal values of the functional $\mathcal J^{(*)}_j.$ Instead, we first provide, in \autoref{fig:variable_d}, for a better understanding a comparison between the simulated data $y,$ for a two layer model, and the forward model. This is shown for different values of $d$ and $r^\dagger_2,$ and the $L^2$-error between the data and the model is calculated, see \autoref{tab:values_rd}. 

\begin{table}
\centering
\begin{tabular}{ |c|c|c|c| } 
\hline
$r^\dagger_2$ & $d$ [mm] & $L^2$-error & color\\
\hline
 \hline
 0.0658 & 0.17 & 0 & blue\\
 0.05 & 0.17 & 17.75 & red\\ 
 0.05 & 0.22 & 135.69 & black\\ 
 0.1 & 0.15 & 433.53 & brown\\ 
 \hline
\end{tabular}
\caption{The different values used for the plot in \autoref{fig:variable_d} and the calculated $L^2$-error. The first line in the table showing the actual values used for the simulated of $y.$}
\label{tab:values_rd}
\end{table}      
The influence of $r^\dagger_2,$ especially for a large value, on the $L^2$-error has a more severe impact compared to the cases where the distance in the model is chosen incorrectly. In order to retrieve the width correctly, we have to use a test functional for suitably chosen refractive index, since the correct is not determined at this stage. In order to minimize the $L^2$-error, the test model with respect to $r^\dagger_2$ (and $n$) should satisfy $|r^\dagger_2| < |\tilde r_2|,$ which guarantees that the second peak of the model is below the one of the data.     
\begin{figure}[htb!]
\centering
\includegraphics[scale = 0.5]{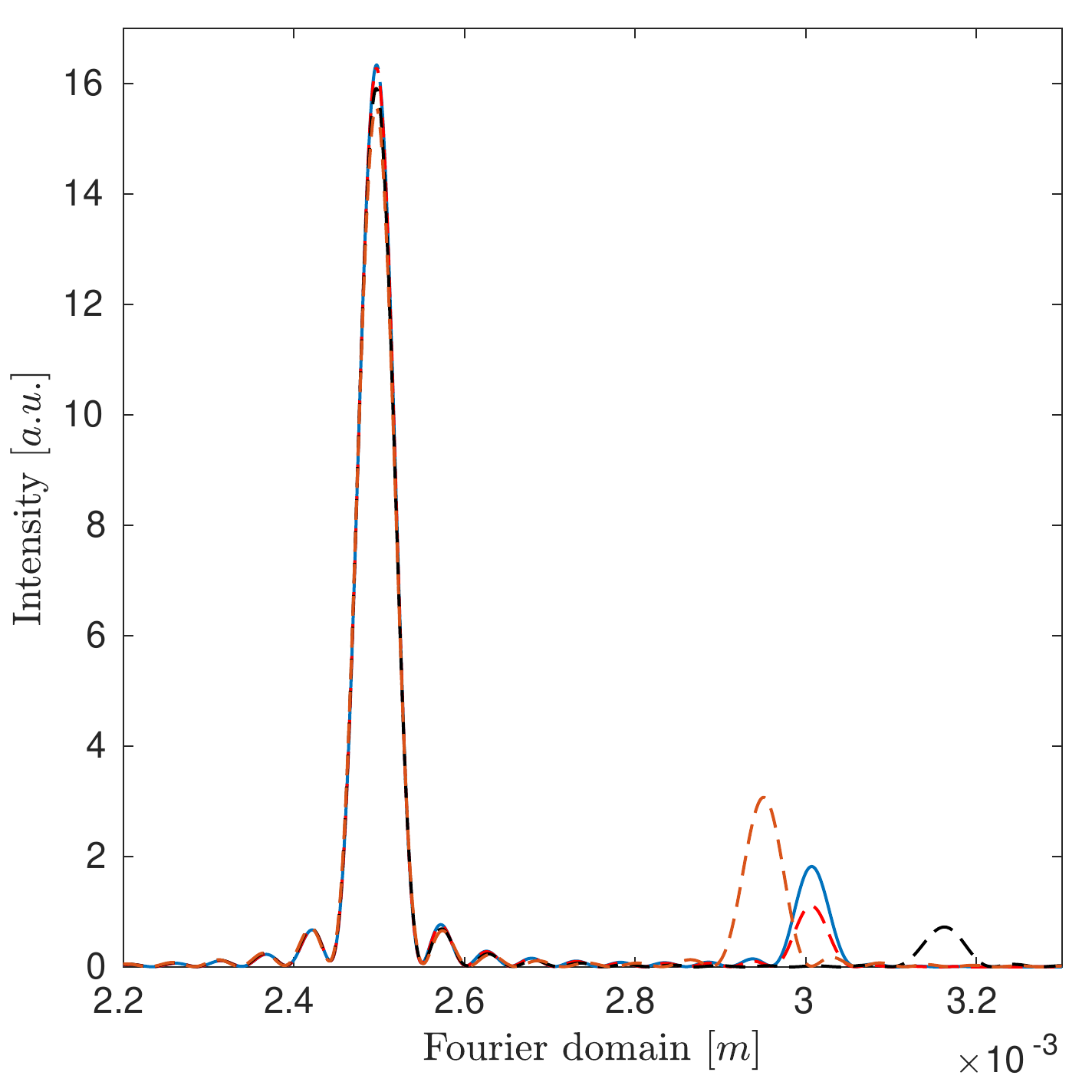}
\vspace{-1em}
\caption{The simulated data $y$ for two reflections in blue is compared with the model for different values of $d$ and $r^\dagger_2$: for a relatively large value $r^\dagger_2$ and small distance $d$ (brown), the actual distance $d$ and small $r^\dagger_2$ (red) and where both values are off (black). }
\label{fig:variable_d}
\end{figure} 

In \autoref{fig:vary_M}, we show the (approximated) functional for the same two layer medium on a grid of values $d.$ From left to right, the number of grid points $M_2$ for the integration used in the functional in \autoref{eq:approx_functional} increases. In order to point out the independence of the existence of a (global) minimum on the refractive index, we plotted for each $M_{2,i},\, i=1,2,3,$ the functional for different values of $r_2^\dagger,$ satisfying $|r^\dagger_2| \leq |\tilde r_2|,$ where $|\tilde r_2|$ is the upper bound for the peak height in the data.  

Clearly, for the actual value $r_2^\dagger = \tilde r_2,$ the functional is zero for all $M_{2,i}.$ By increasing the number of grid points we additionally obtain a sharpening effect which pronounces the minimum even better. Further, we see that for $r_2^\dagger$ with the wrong sign and where the simulated peak is ultimately on the level of the sideloops of the sinc-functions, the brown curve in \autoref{fig:vary_M}, the minimum is shifted slightly for larger values of $M_{2,i}.$     
\begin{figure}[htb!]
\centering
\includegraphics[scale = 1.65]{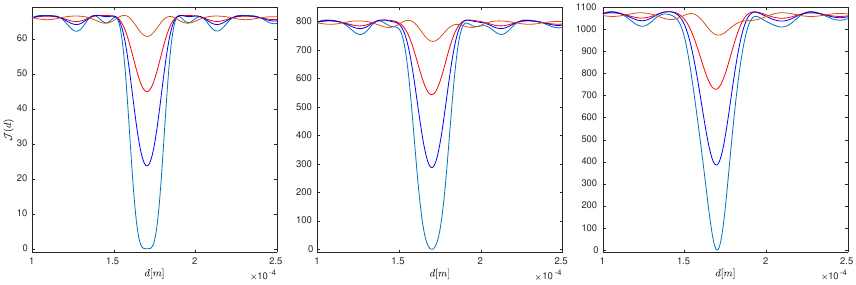}
\vspace{-1em}
\caption{The functional $\mathcal J^*,$ shown on a grid for values of $d,$ is presented for different values of $r^\dagger_2 = \{\tilde r_2, 0.025, 0.04, -0.02\},$ where $\tilde r_2 = 0.0658$ (blue, red, dark blue, brown) and for different numbers $M_2 = \{1, 301, 601\}$ (left to right). }
\label{fig:vary_M}
\end{figure} 
The location of the global minimum of the functional for each $M_{2,i}$ and every $r_{2,l}^\dagger,\, l\in\{1,\dots,4\},$ is presented in the first three lines of \autoref{tab:d_pos}. There, we see that for an intermediate number of grid points, the localization of the minimum position is not shifted for resonable values of $r_2^\dagger.$ 
\begin{table}[hbt!]
\centering
\begin{tabular}{ |c|c|c|c|c|c| } 
\hline
& $r^\dagger_{2,1} = \tilde r_2$ & $r_{2,2}^\dagger$ & $r_{2,3}^\dagger$ & $r^\dagger_{2,4}$& \\
\hline
\hline
$d$ [mm] & 0.17 & 0.1703 & 0.1704 & 0.1705 & $M_{2,1}$\\ 
$d$ [mm] & 0.17 & 0.17 & 0.17 & 0.1707 & $M_{2,2}$\\ 
$d$ [mm] & 0.17 & 0.1691 & 0.1694 & 0.1707 & $M_{2,3}$ \\ 
 \hline
 \hline
$d$ [mm] & 0.1699 & 0.1699 & 0.17 & 0.1706 & $M_{2,1}$\\ 
$d$ [mm] & 0.17 & 0.169 & 0.1694 & 0.1707 & $M_{2,2}$\\ 
\hline
\end{tabular}
\caption{The minimum values of the functional $\mathcal J^*_2$ for different reflection coefficients $r_2^\dagger$ and number of grid points $M_{2,i}.$ The first three lines correspond to the case where the data satisfies the range condition and which are plotted in \autoref{fig:vary_M}. The two bottom lines show the minimas of the functional, where an additional term was added to the data.}
\label{tab:d_pos}
\end{table}
Finally, we consider the case where the data is disturbed 
\[
y_{j,m} = |\Lambda_{j,m} + \tilde r_j \Gamma^*_{j,m} + \Phi|^2, \quad 1\leq m\leq M_j,
\]
for some $\Phi$ representing possible noise or a reflection from deeper inside the sample. We use the forward model 
\[
|\Lambda_{j,m} + r_j^\dagger(n) \Gamma^*_{j,m}|^2 
\]
for the functional $\mathcal J^*_j$ in \autoref{eq:approx_functional}, and we still obtain a unique global minimum. For the simulation of the data, we used for $\Phi$ the forward model (for an additional layer reflection) with $\tilde r_3 = -\tilde r_2$ and distance $d_2 = d_1.$ In \autoref{fig:functional_disturbed}, we plot the functional for different values of $r_2^\dagger$ and $M_{2,i}\in\{150,300\}.$ In \autoref{tab:d_pos} the two bottom rows show the position of the global minimum for the disturbed data. Here, again, an intermediate value of $M_2$ is the most effective one. 
\begin{figure}[htb!]
\centering
\includegraphics[scale = 1.65]{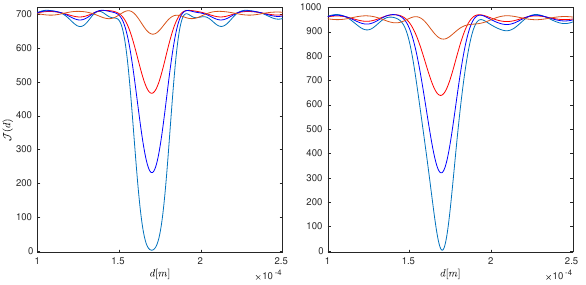}
\vspace{-1em}
\caption{The functional $\mathcal J^*,$ shown on a grid for values of $d,$ is presented for different values of $r^\dagger_2 = \{\tilde r_2, 0.025, 0.04, -0.02\}$ (blue, red, dark blue, brown) and for different numbers $M_2 = \{301, 600\}$ (left to right). }
\label{fig:functional_disturbed}
\end{figure} 

Aside from the global minimum, which we obtain in every case above, there exist also local minima. These occur if the peak overlaps with the sideloops of the data-sinc-function and therefore yield a smaller $L^2$-error. Hence, for the minimization (process) of the functionals, the initial guess is crucial. Otherwise the solution is stuck in the neighbourhood of one of these local minima.

\section{Numerical Experiments}
\label{sec:numerics}

\subsection*{OCT setup and phantom description}
In the following examples we illustrate the applicability of the proposed method for both simulated and experimental data. The OCT experimental data, shown in \autoref{fig:data}, was generated by a swept-source OCT system with a laser light source centered at the wavelength $\lambda_0 = 1300$nm. The bandwidth around the central wavelength ranges from approximatively $1282.86$nm to $1313.76$nm. During a single experiment the object was raster scanned on a grid of $1024 \times 1024$ points in horizontal direction and for each raster scan a spectrum consisting of $1498$ data points, equally spaced with respect to the wavenumber, was recorded. Due to the relatively small bandwidth of the laser, the approximation that $n$ does not dependent on the wave number is valid. For the reconstruction, a single depth profile was chosen from the data set and no averaging has been done. 
Additionally, for the calibration of all necessary system parameters, introduced in \cite{VesKraMinDreElb21}, the object was shifted multiple times along the axial direction, where for each position a full measurement was recorded.

\begin{figure}[htb!]
\centering
\includegraphics[scale = 1.1]{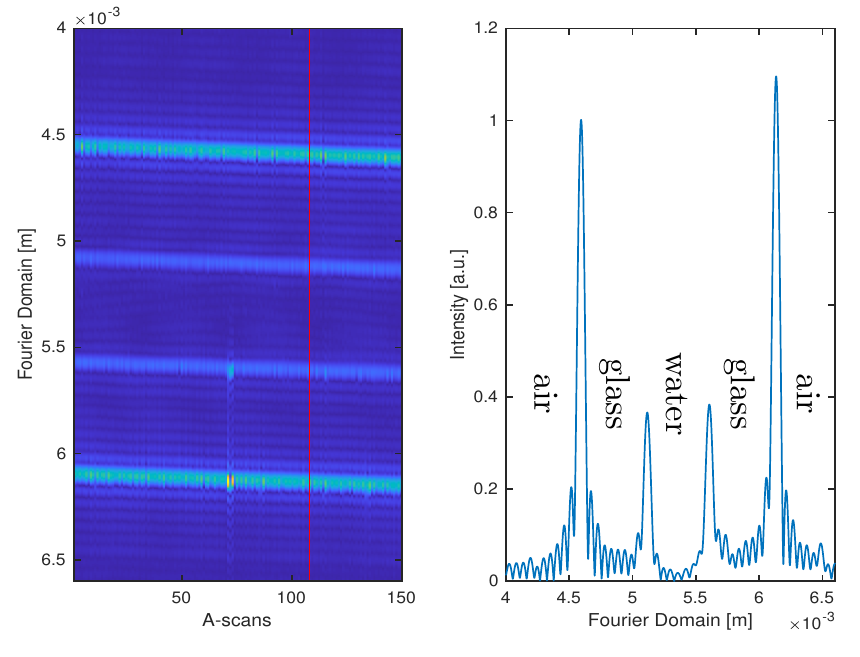}
\vspace{-1em}
\caption{The OCT experimental data for a glass-water-glass sample produced by a $1300nm$ system. A cross-sectional image, a B-scan, (left) showing $150$ depth profiles (A-scans) of the object. A single depth scan (right), which is taken along the red line of the B-scan, shows the single layers of the object. The intensity in the A-scan is normalized such that the first peak has magnitude one. }
\label{fig:data}
\end{figure} 

In order to examine the correctness of the recovered values, the object needs to consist of materials where the ground truths of the refractive indices, on the used spectrum, are available. The sample is a three-layer medium where two coverglasses enclose water. The refractive indices of both materials, at the central wavelength, are given by $n_g \approx 1.5088$ and $n_w \approx 1.3225.$ The thickness of every layer is estimated between $0.14$mm and $0.19$mm. 

The simulated data were generated for the same three-layer sample with refractive indices $n_0 = 1,\,n_1 = 1.5088,\, n_2 = 1.3225,\, n_3 = 1.5088$ and widths $d = [0.174,\,0.186,\,0.173]$mm by using the forward model presented in \autoref{sec:mathmodel}. 
Given the list of system and object parameters described in \autoref{as:three}, the formula for the reflected field \autoref{eq:scatfield} was implemented where the domain of accepted wave directions, see \autoref{eq:accepwavevec}, was replaced by its discretized version. 

In addition we use a sample showing characteristica in lateral directions, which are then visible in the outcome for different raster scans across the object. On a grid of $20\times 20,$ we define the object for every $(l,m)\in\{1,\dots,20\}^2$ as a two layer object with $n(l,m) = [1,\, 1.5088,\, \hat n(l,m)],$ where the distribution $\hat n$ is shown in \autoref{fig:3D_data}. For every grid element, we collected the simulated measurement data. For the reconstruction $5\%$ uniformly distributed noise was added to the data. 

Multiple reflections, due to their minor contribution to the final measurement, have been omitted. 

\subsection*{Inverse problem and minimization}
For the inverse problem, the functionals $\mathcal J_j$ and $\mathcal J^*_j$, based on the direct model, defined in \autoref{eq:min_func} and \autoref{eq:approx_functional} respectively, were implemented. The discretized integral uses $M_j = M = 401$ equally spaced grid points in Fourier domain, which have been chosen symmetrically around every exposed local maximum in the data.
\begin{description}
\item[Simulated vs. Real Data]
In \autoref{sec:inverse_EaU}, we discussed the minimization problem for the data $y_j,\, 1\leq j\leq J+1,$ which so far, has been considered as an element of the range of the forward operator (plus a certain noise level $\delta$). However, experimental data fails to satisfy this relation because of the unknown intensity in the focal plane: \\
In \autoref{as:one} we have set the amplitude of the Gaussian distribution in our forward model equal to one, which further yields that the maximal intensity of the incident field in \autoref{eq:solution}, that is $\big|E^{(0)}\big|^2$ in the focal spot $x_3 = r_0,$ is also one. 
For the experimental data the maximal intensity represents the power of the laser light in the focal spot and is in general an unknown quantity. We consider this quantity as real positive number $ Q_0.$ Even for a single layer reflection, the factor $Q_0$ causes incorrect reconstructions if it is not included in the model. Let the data be given by 
\[
\tilde y = Q_0 y, \quad\text{for}\quad y = |r^\dagger(\tilde n_1) \Gamma^*_{1,m_0}|^2,
\] 
for a single grid point $z_{1,m_0}.$ Then from \autoref{thm:unique_min}, we deduce that
\[
|r^\dagger(n) \Gamma^*_{1,m_0}|^2 = Q_0 |r^\dagger(\tilde n_1) \Gamma^*_{1,m_0}|^2,
\]         
which then yields $|r^\dagger(n)| = \sqrt{Q_0} |r^\dagger(\tilde n_1)|.$ 

We correct the mismatch between the model and the experimental data by the quantifying $Q_0$ from the first air-glass reflection. Hereby, we use the coverglass plate as a calibrational layer. We recover instead of the refractive index $n_1,$ which we assume in this case to be determined perfectly before, the quantity $Q_0$ as the quotient between the data and the forward model for $j=1$ in \autoref{eq:min_func}.

\item[Minimization in Two Steps:] In \autoref{sec:layer} (which is carried out in \autoref{sec:inverse_EaU}) we proposed a layer-by-layer method where in each step a single pair of parameters $(n_j,d_{j-1}),\, 1\leq j\leq J+1,$ is recovered. Each step itself is then again split into two steps. After the minimum of the functional $\mathcal J^*_j$ is determined, the output is used as an input in form of an initial guess for the actual minimization problem corresponding to the functional $\mathcal J_j.$ In \autoref{fig:comp_sim} and \autoref{fig:comp_real} a comparison of the two functionals (for the minimum value of $d$) for every reconstruction step is provided for simulated and experimental data, respectively. 
\begin{figure}
\centering
\includegraphics[scale = 0.7]{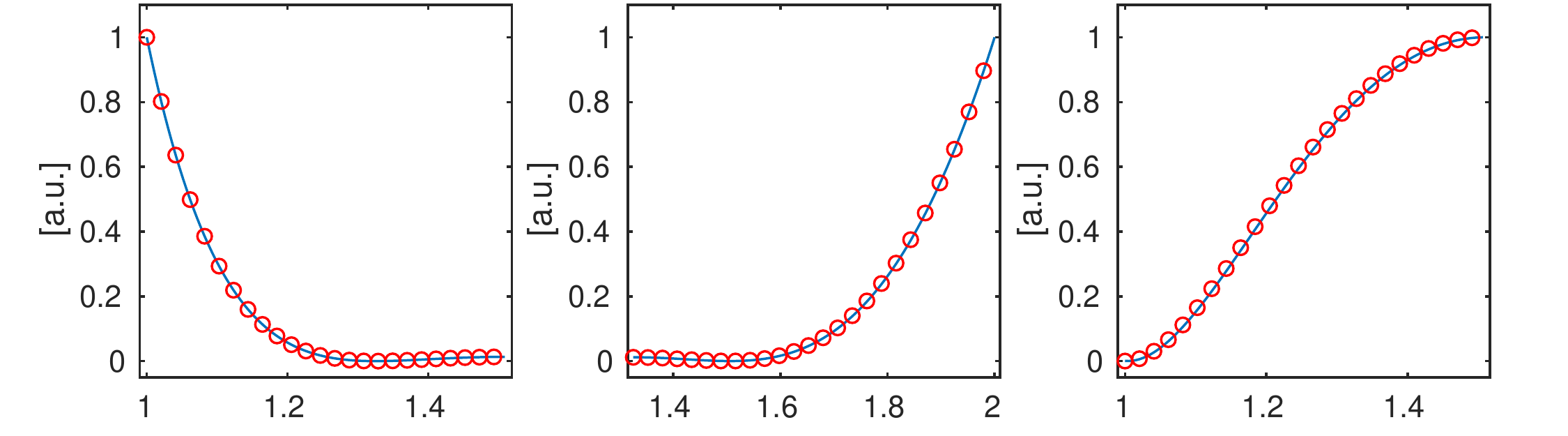}
\caption{Comparison between the functionals $\mathcal J_j$(blue line) and $\mathcal J^*_j$(red circles). The $x$-axis shows possible values of the refractive index $n_j$ in each step $j=1$ to $j=3$ for the reconstruction for simulated data. The order is from left to right. }
\label{fig:comp_sim}
\end{figure}

\begin{figure}[htb!]
\centering
\includegraphics[scale = 0.7]{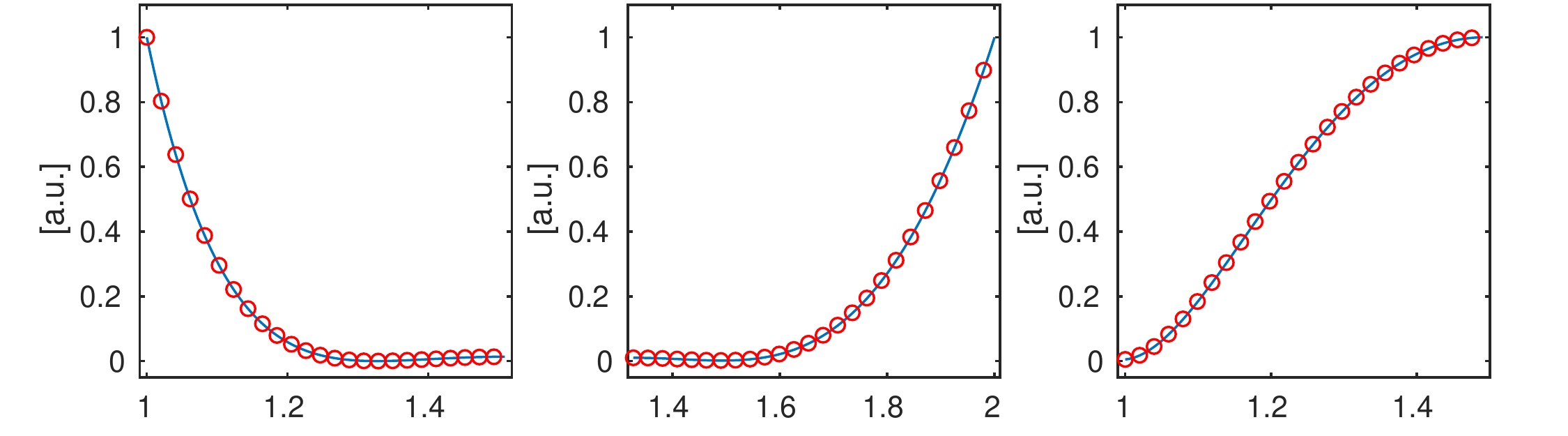}
\caption{Comparison between the functionals $\mathcal J_j$(blue line) and $\mathcal J^*_j$(red circles). The $x$-axis shows possible values of the refractive index $n_j$ in each step $j=1$ to $j=3$ for the reconstruction for experimental data. The order is from left to right.}
\label{fig:comp_real}
\end{figure}  
\end{description}

\subsection*{Results}
The outcome of both minimization problems, for simulated and experimental data, are presented in this section. In \autoref{fig:func_sim} we plot the functionals $\mathcal J_j,\,j=1,2,3,$ on a $(n,d)$-grid on the half of the admissible values $\mathcal A_j$ where the global minimum is located. 
\begin{figure}[htb!]
\centering
\includegraphics[scale = 1.5]{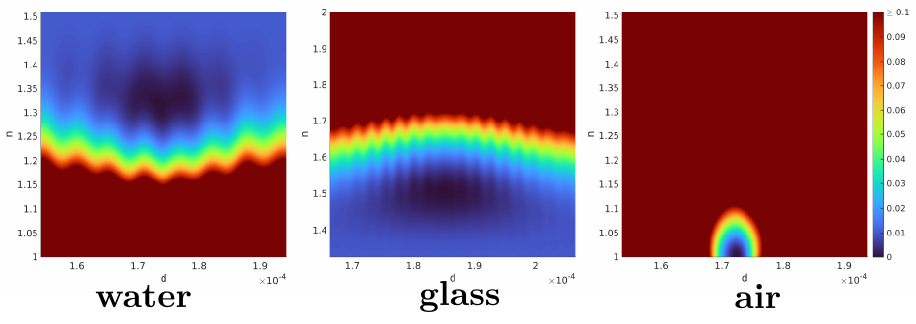}
\caption{The functionals $\mathcal J_j$ (for simulated data) for each reconstruction step evaluated on a grid of $150\times 100$ values for the refractive index $n$ and the width $d.$ }
\label{fig:func_sim}
\end{figure}  

For the experimental data, the functionals evaluated over a $(n,d)$-grid are shown in \autoref{fig:func_real}. Additionally, a contour plot of the close neighbourhoods of each pair of minima are shown in \autoref{fig:contour}. The reconstructed values are plotted in \autoref{fig:results}. The laterally reconstructed refractive index distribution is shown in \autoref{fig:3D_data}.

\begin{figure}
\centering
\includegraphics[scale = 1.5]{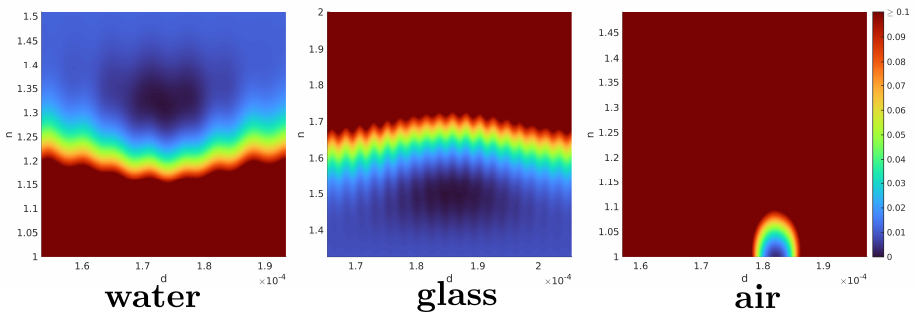}
\caption{The functionals $\mathcal J_j$ (for experimental data) for each reconstruction step evaluated on a grid of $150\times 100$ values for the refractive index $n$ and the width $d.$ }
\label{fig:func_real}
\end{figure} 
 
\begin{figure}
\centering
\includegraphics[scale = 0.7]{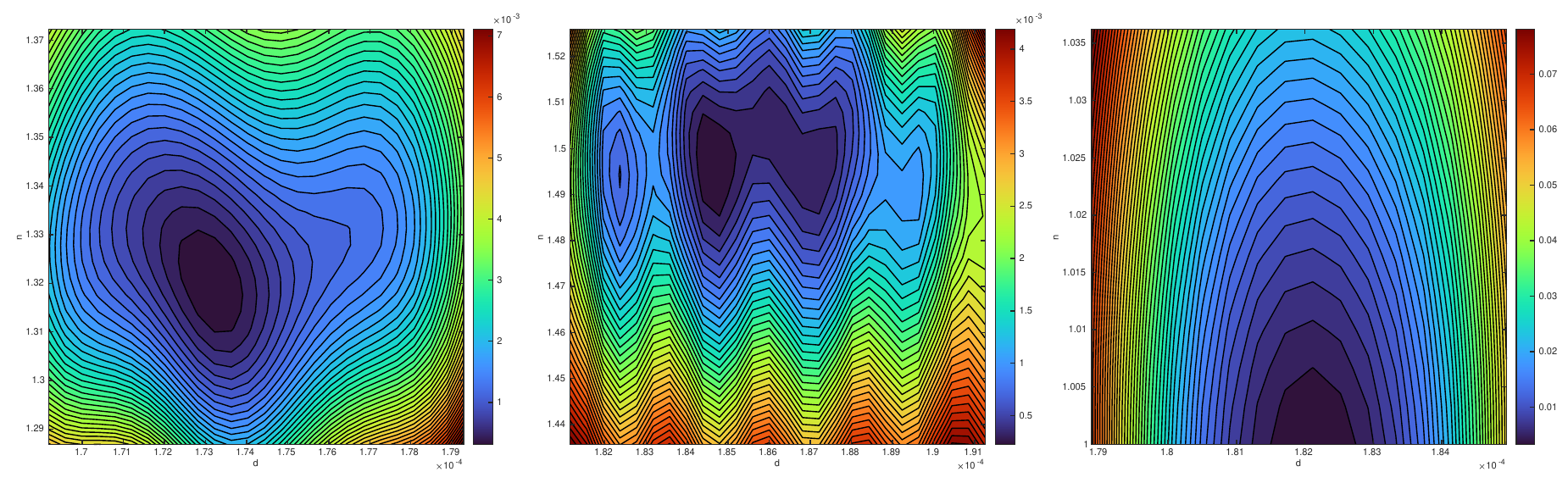}
\caption{A contour plot of the functionals $\mathcal J_j$ (for experimental data) for each reconstruction step evaluated within a close neighbourhood of each pair of minima $(n,d).$ }
\label{fig:contour}
\end{figure}

\begin{figure}[hbt!]
\centering
\includegraphics[scale = 0.7]{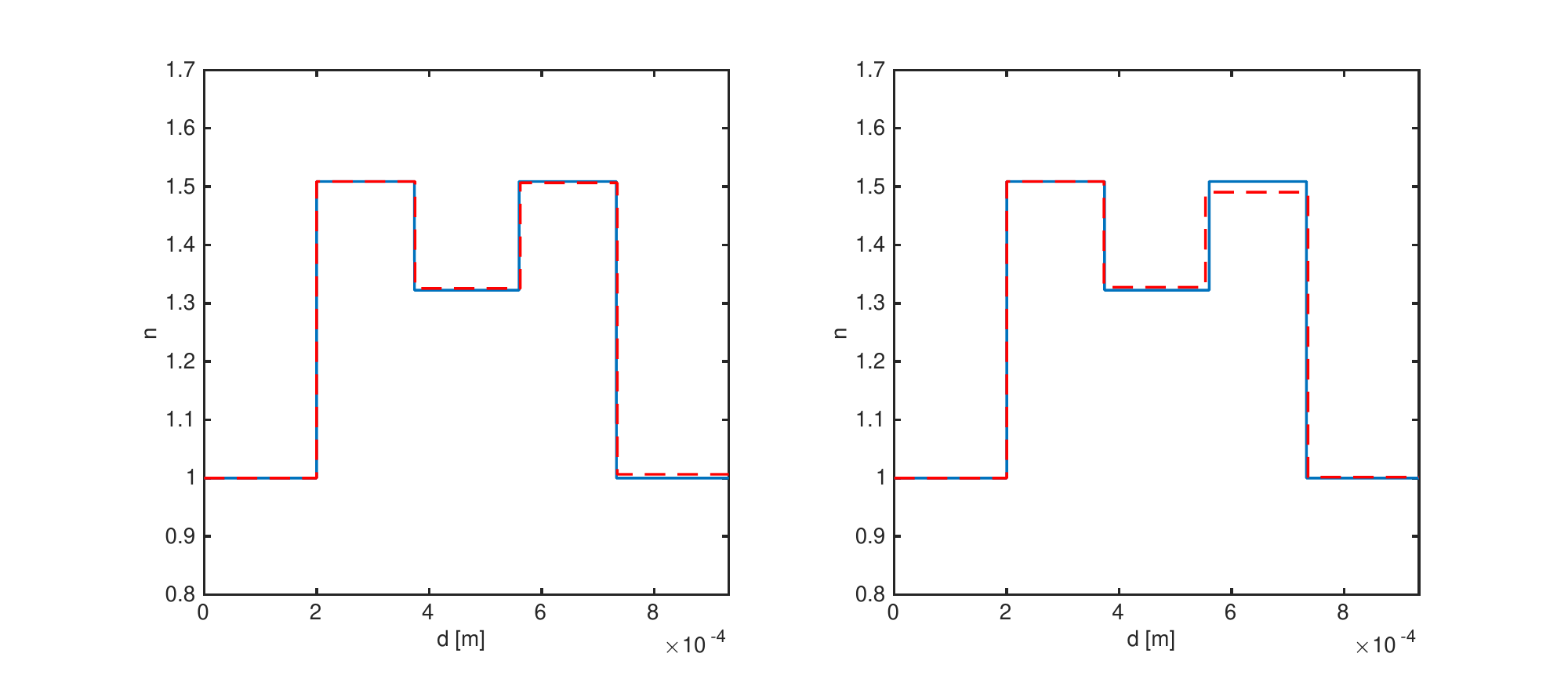}
\caption{A comparison of the reconstructed values (red) with the ground truth (blue). The reconstruction for the simulated data without noise (left) and the experimental data (right) match nicely the ground truth.}
\label{fig:results}
\end{figure}

\begin{figure}
\centering
\includegraphics[scale = 0.85]{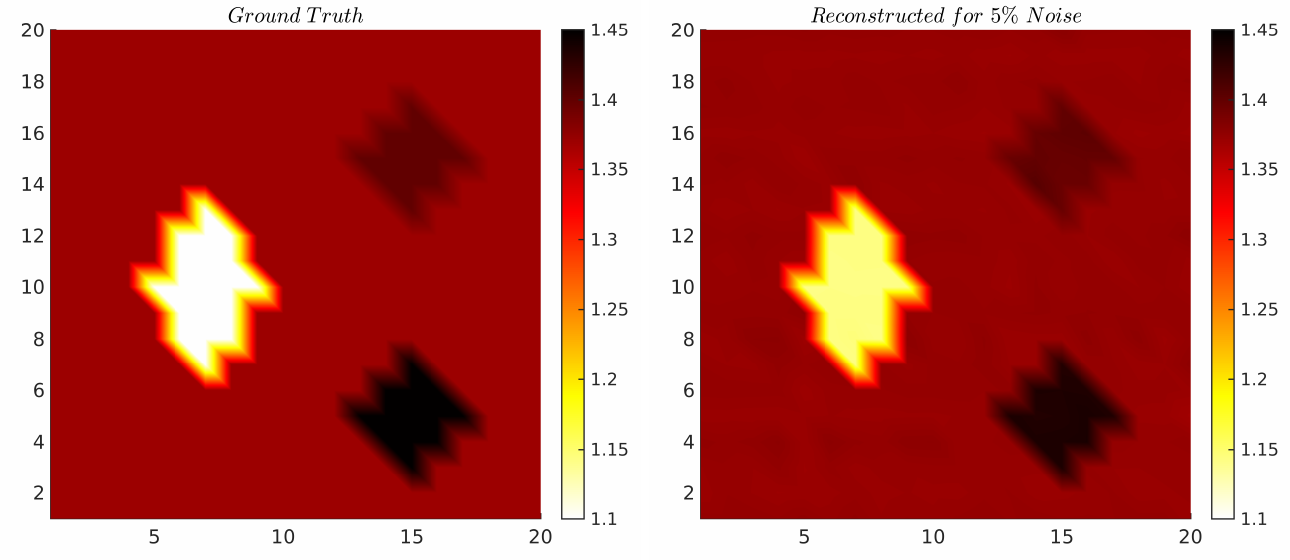}
\caption{The ground truth refractive index distribution $\hat n$ on a $20\times 20$ grid, showing the discrete versions of two circles $\hat n_{C_1} = 1.45,\, \hat n_{C_2} = 1.4$ and an ellipse $\hat n_{E} = 1.1$ surrounded by a homogeneous medium with refractive index $\hat n_B = 1.37$ (left). The reconstructed distribution from simulated data with $5\%$ (random) noise (right). }
\label{fig:3D_data}
\end{figure}

\newpage
\section*{Conclusion}
The inverse problem in quantitative optical coherence tomography for layered media has been discussed. Based on a Gaussian beam forward model a discrete least squares minimization problem for the reconstruction of the refractive index and the thickness of each layer is formulated. Existence and uniqueness of solutions for the minimization problem are discussed. The method is validated by numerical examples considering the refractive index reconstruction from a three-layer object from both simulated and experimental data.

\section*{Acknowledgement}
This work was made possible by the greatly appreciated support of the Austrian Science Fund (FWF) via the special research programme SFB F68 ``Tomography Across the Scales'':
Peter Elbau and Leopold Veselka have been supported via the subproject F6804-N36 ``Quantitative Coupled Physics Imaging'', Lisa Krainz and Wolfgang Drexler have been supported via the subproject F6803-N36 ``Multi-Modal Imaging''.

\section*{References}
\printbibliography[heading=none]

\end{document}